\documentclass[12pt]{article}

\usepackage{amsfonts,amsmath,amsthm,latexsym,amssymb,fullpage,amscd}
\usepackage[mathscr]{eucal}
\usepackage{tikz}

 at 16 true pt
 at 13  true pt
 at 12 true pt

\newcommand{\R}{\mathbb{R}}
\newtheorem{theorem}{Theorem}[section]
\newtheorem{lemma}[theorem]{Lemma}
\newtheorem{corollary}{Corollary}[theorem]
\newtheorem{definition}{Definition}[section]

\title{Hessian determinants and averaging operators over surfaces in $\R^3$}

\author{Michael Greenblatt}
\date{\today}

\newcommand\blfootnote[1]{%
  \begingroup
  \renewcommand\thefootnote{}\footnote{#1}%
  \addtocounter{footnote}{-1}%
  \endgroup
}
\begin{document}
\maketitle
\begin{abstract} 

We prove $L^p({\mathbb R}^3)$ to $L^p_s(\R^3)$ Sobolev improvement theorems for local averaging operators over real analytic surfaces in $\R^3$. 
In a sense made precise in the paper, the set of $(p,s)$ for which we prove $L^p(\R^3)$ to $L^p_s(\R^3)$ boundedness is optimal up to endpoints for generic such operators. Using an interpolation argument in conjunction with  these $L^p(\R^3)$ to $L^p_s(\R^3)$ results we obtain an $L^p(\R^3)$ to $L^q(\R^3)$ improvement theorem, and the set of exponents $(p,q)$ obtained will also generically be optimal up to endpoints.
 The advantage the methods of this paper have over those of the author's earlier papers 
is that the oscillatory integral methods of the earlier papers, closely tied to the Van der Corput lemma, allow one to 
only prove 1/2 of a derivative of surface measure Fourier transform decay, while the methods of this paper, when combined with appropriate resolution of singularities methods, allow one to go up to the maximum possible 1 derivative. This allows us to prove the stronger sharp up to endpoints results. 

\end{abstract}
\blfootnote{This work was supported by a grant from the Simons Foundation.}

\section{ Introduction and theorem statements } 

We prove $L^p({\mathbb R}^3)$ to $L^p_s({\mathbb R}^3)$ Sobolev improvement theorems for local averaging operators over real analytic surfaces in ${\mathbb R}^3$. For generic operators, in a sense we will make precise, the set of $(p,s)$ for which we prove $L^p({\mathbb R}^3)$ to $L^p_s{(\mathbb R}^3)$ boundedness is optimal up to endpoints. Using an interpolation argument in conjunction with  these $L^p({\mathbb R}^3)$ to $L^p_s({\mathbb R}^3)$ results we obtain an $L^p({\mathbb R}^3)$ to $L^q({\mathbb R}^3)$ improvement theorem, and the set of exponents $(p,q)$ obtained will also generically be optimal up to endpoints.

\subsection {Definitions and key concepts}

We consider the operator, initially defined on Schwartz functions on $\R^3$,  defined by
\[ Tf(x_1,x_2, x_3) = \int_{\R^2} f(x_1 - t_1, x_2 - t_2, x_3 - S(t_1,t_2))\, \phi(t_1,t_2)\,dt_1dt_2 \tag{1.1}\]
Here $\phi(t_1,t_2)$ is a bump function supported near the origin and $S(t_1,t_2)$ is a real analytic function, not identically zero, defined on a 
neighborhood of the origin that contains the support of $\phi(t_1,t_2)$, which satisfies
\[S(0,0) = 0 {\hskip 1 in} \nabla S(0,0) = (0,0) \tag{1.2}\]
Since our theorems will be invariant under invertible linear transformations, we do not lose generality by assuming $(1.2)$. 

A key concept in our theorems
and proofs will be the following. Suppse $g(t_1,t_2)$ is a real analytic function defined on a neighborhood of the origin which satisfies 
$g(0,0) = 0$ and which is not identically zero. 
Then using resolution of singularities, (we refer to Chapters 6-7 of [AGV] for more details) one can show that there are $0 < \eta \leq 1$,
an integer $k = 0$ or $1$ and a neighborhood $U$ of the origin in $\R^2$, such that if $V \subset U$ is a neighborhood of the origin then for
some positive constants $b_V, c_V$ one has the following for all $0 < \delta < {1 \over 2}$.
\[b_V\delta^{\eta} |\ln \delta|^k < m(\{(t_1,t_2) \in V:  |g(t_1,t_2))| < \delta\}) < c_V \delta^{\eta} |\ln\delta|^k \tag{1.3}\]
Here $m$ denotes Lebesgue measure. 
The exponent $\eta$ in $(1.3)$ can be described succinctly as the supremum of the numbers $\epsilon > 0$ for which $|g|^{-\epsilon}$ is locally
integrable in any sufficiently small neighborhood of the origin. 

When $(1.2)$ holds, the maximum possible value of $\eta$ is $1$, which occurs in the nondegenerate case, i.e. when the Hessian determinant of 
$S$ is nonzero at the origin. In this situation, the optimal $L^p$ to $L^q_s$ results follow immediately from interpolating with 
 the well-established
$L^2$ to $L^2_1$ and $L^p$ to $L^q$ boundedness properties for nondegenerate surfaces. As a result, in this paper we will always assume that
the Hessian determinant of $S(t_1,t_2)$, which we denote by $H(t_1,t_2)$, satisfies
\[ H(0,0) = 0 \tag{1.4}\]

We will also assume that $H(t_1,t_2)$ is not identically zero. The case where $H(t_1,t_2)$ is identically zero effectively means that $S(t_1,t_2)$ behaves like a  
function of one variable. In fact, it can be shown that if $S(t_1,t_2)$ is a polynomial, if $H(t_1,t_2)$ is identically zero then after a linear transformation
$S(t_1,t_2)$ does indeed becomes a function of one variable, in which case the Sobolev improvement theorems known for curves in $\R^2$  (see
[C1] and [Gra]) imply optimal results for the surfaces. 

\subsection {$L^p$ to $L^p_s$ theorems}

We now come to the main theorem of this paper; the other results will follow from combining this result with other results, using appropriate interpolation
arguments. 

\begin{theorem} Suppose $(1.2)$ and $(1.4)$ hold and neither $S(t_1,t_2)$ nor $H(t_1,t_2)$ are identically zero. Let $\eta$ be the exponent in $(1.3)$ corresponding to $S(t_1,t_2)$, and let $\eta'$ be the  exponent in $(1.3)$ corresponding to $H(t_1,t_2)$. There is a neighborhood $U$ of the origin
such that if $\phi(t_1,t_2)$ is supported in $U$, then
$T$ is bounded from $L^p(\R^3)$ to $L^p_s(\R^3)$ for $({1 \over p}, s)$ in the interior of the trapezoid in the $xy$ plane bounded by the lines $y = 0,
 y = \min(\eta,{2\eta' \over 1 + 2\eta'}),\, y = 2x,$ and $y = 2 - 2x$. 
If $\eta \leq {2\eta' \over 1 + 2\eta'} $, then this result is sharp up to endpoints;
 if $\phi$ is nonnegative with $\phi(0,0) > 0$ and $({1 \over p}, s)$ is not in the closed trapezoid, then $T$ is not bounded from $L^p(\R^3)$ to $L^p_s(\R^3)$.
 \end{theorem}

We address the sharpness statement of Theorem 1.1. One never has $L^p$ to $L^p_s$ boundedness for $T$ when $({1 \over p}, s)$ is strictly above the line $y = 2 - 2x$; this can be verified by testing $T$ on
bump functions supported on balls of radius $r$ for $r \rightarrow 0$. By duality, one therefore never has $L^p$ to $L^p_s$ boundedness for $T$ when
$({1 \over p}, s)$ is strictly above the line $y = 2x$. If $1 < p < \infty$, one also never has $L^p$ to $L^p_s$ boundedness for $s > \eta$; this can 
be verified by testing
on functions of the form $\psi(x_1,x_2,x_3)|x_3|^a$ for $0 < a < 1$ and $\psi$ a nonnegative bump function satisfying $\psi(0,0,0) > 0$. The details of this argument are worked out at the end of [G1].
 
Thus if $\eta \leq {2\eta' \over 1 + 2\eta'} $, Theorem 1.1 provides optimal $L^p$ Sobolev improvement up to endpoints. As we will see in section 2.1, there is a sense
in which generically one actually has $\eta = {2\eta' \over 1 + 2\eta'} $. Thus Theorem 1.1 gives boundedness that is sharp up to endpoints in
such generic scenarios. If ${2\eta' \over 1 + 2\eta'}  < \eta$, one can interpolate Theorem 1.1 with results in the literature to obtain stronger results. If $\eta \leq 
{1 \over 2}$, one can interpolate with Theorem 1.2 of [G2], and if $\eta > {1 \over 2}$ one can interpolate with the $L^2$ to $L^2_{\eta}$ boundedness
that is known to hold for all surfaces; this boundedness is equivalent to the uniform decay rates for the Fourier transforms of surface measures and was
proven by Duistermaat in [Du]. The result of these interpolations is the following.

\begin{theorem} Suppose we are in the setting of Theorem 1.1 and ${2\eta' \over 1 + 2\eta'}  < \eta$. There is a neighborhood $U$ of the origin
such that if $\phi(t_1,t_2)$ is supported in $U$, then
\begin{enumerate}
\item If $\eta \leq {1 \over 2}$, then $T$ is bounded from $L^p(\R^3)$ to $L^p_s(\R^3)$ for $({1 \over p}, s)$ in the interior of the polygon with vertices 
$(0,0), ({\eta' \over 1 + 2\eta'}, {2\eta' \over 1 + 2\eta'}), (\eta,\eta), (1-\eta, 1 - \eta), (1 - {\eta' \over 1 + 2\eta'}, {2\eta' \over 1 + 2\eta'}),$ and
$(1,0)$.
\item If $\eta > {1 \over 2}$, then $T$ is bounded from $L^p(\R^3)$ to $L^p_s(\R^3)$ for $({1 \over p}, s)$ in the interior of the polygon with vertices 
$(0,0), ({\eta' \over 1 + 2\eta'}, {2\eta' \over 1 + 2\eta'}), ({1 \over 2}, \eta), (1 - {\eta' \over 1 + 2\eta'}, {2\eta' \over 1 + 2\eta'})$, and
$(1,0)$.
\end{enumerate}
\end{theorem}

There has been a lot of work done on boundedness properties of averaging operators on function spaces. For curves in $\R^2$, [Se] proves 
$L^p_{\alpha}$ to $L^q_{\beta}$ boundedness theorems that are complete up to endpoints. For translation invariant averaging operators, $L^2$ to $L^2_{\beta}$ Sobolev space improvement is equivalent to a surface measure Fourier transform 
decay rate estimate. For surfaces in three dimensions, the stability theorems of Karpushkin [Ka1] [Ka2], when combined with [V], give such sharp decay rate
 results for real analytic surfaces. Extensions to finite type smooth surfaces appear in [IkKM]. There are also some results for convex (non necessarily smooth) surfaces such as [R1][R2]. For general $p$, the paper
 [St]  considers Sobolev estimates for Radon transforms in a quite general setting, focusing attention on singular density functions. We also mention the
 papers [HeHoY1] [HeHoY2] in this area. The author has also 
 written several previous papers with results on this subject, including [G1][G2][G3][G6].

\subsection{$L^p$ to $L^q$ theorems}

It follows from Theorem 1.2 of [G3] that for any $1 < p \leq q < \infty$ and any $r > -1$, the operator $T$ is bounded from $L^p$ to $L^q_r$. 
For a sequence of $({1 \over p_n}, {1 \over q_n}, r_n)$ approaching $(1,0,-1)$, we interpolate this fact with Theorems 1.1 and 1.2, obtaining an
$L^p$ to $L^q$ boundedness region for $1 <p < q < \infty$. Theorems 1.1 and 1.2 translate as follows.

\begin{theorem} Suppose we are in the setting of Theorem 1.1 and ${2\eta' \over 1 + 2\eta'} \geq \eta$.
 There is a neighborhood $U$ of the origin
such that if $\phi(t_1,t_2)$ is supported in $U$, then $T$ is bounded from $L^p(\R^3)$ to $L^q(\R^3)$ for $({1 \over p}, {1 \over q})$ in the interior of the trapezoid 
in the $xy$ plane bounded by the lines $y = x$, $y = x - {\eta\over \eta + 1}$, $y = {x \over 3}$, and $y = 3x - 2$.
\end{theorem}

\begin{theorem} Suppose we are in the setting of Theorem 1.1 and ${2\eta' \over 1 + 2\eta'} < \eta$. Let $R$ be the open region 
given by Theorem 1.2, and let $R'$ be the image of $R$ under the map $(x,y) \rightarrow ({x + y  \over y + 1}, {x \over y + 1})$. There is a
 neighborhood $U$ of the origin such that if $\phi(t_1,t_2)$ is supported in $U$ and $({1 \over p}, {1 \over q})$ is in $R'$ then $T$ is bounded from
  $L^p(\R^3)$ to $L^q(\R^3)$. 
\end{theorem}

There have been quite a few $L^p$ to $L^q$ boundedness theorems for averaging operators over hypersurfaces.
Surfaces with nonvanishing Gaussian curvature are analyzed in [L][St2][Str]. 
The situation where $S(t_1,t_2)$ is a homogeneous or mixed homogeneous function has been considered in [DZ] [FGoU1] [FGoU2] [ISa] [Sch].
Convex surfaces of finite line type are dealt with in [ISaSe]. We will have more to say about such surfaces later in the paper. Also, there have been papers considering averaging operators with a damping function, often related to the Hessian determinant. We mention [Gr] and [O] as examples.

For the setting of Theorem 1.3, one has sharpness as follows. By testing on functions $f(x)$ that vanish near the
origin and are equal to $|x|^{-\alpha}$ for sufficiently large $|x|$, one can see that $T$ cannot be bounded from $L^p$ to $L^q$ for $q < p$. By testing
on characteristic functions of balls with radii tending to zero, one can show that $T$ cannot be bounded from $L^p$ to $L^q$  for $({1 \over p}, 
{1 \over q})$ below the line $y = 3x - 2$; by duality the same is true for the line $y = {x \over 3}$. 

There remains the line $y = x - 
{\eta \over \eta + 1}$. One generically does not have $L^p$ to $L^q$ boundedness for  $({1 \over p}, {1 \over q})$ below this line. To understand
in what sense this is true, we need to introduce some notions involving Newton polygons and related matters. To this end, let 
$g(x,y)$ denote a power series in $x^{1 \over N}$ and $y$ for some positive integer $N$, and write $g(x,y) = \sum_{a,b} g_{a,b}
x^a y^b$.

\begin{definition}
 For any $(a,b)$ for which $g_{a,b} \neq 0$, let $Q_{ab}$ be the
quadrant $\{(x,y) \in \R^2: 
x \geq a, y \geq b \}$. Then the {\it Newton polygon} $N(g)$ of $g(x,y)$ is defined to be 
the convex hull of the union of all $Q_{ab}$.  
\end{definition}

The boundary of $N(g)$ consists of finitely many (possibly none) bounded edges of negative slope
as well as an unbounded vertical ray and an unbounded horizontal ray. We next have the following definitions.

\begin{definition} For a bounded edge $e$ of $N(g)$, $g_e(x,y)$ denotes the polynomial $\sum_{(a,b) \in e} g_{a,b}x^a y^b$.
\end{definition}

\begin{definition} For a bounded edge $e$ of $N(g)$, $o(e)$ denotes the maximum order of any zero of the polynomial $g_e(1,y)$ other than $y = 0$. 
In the case there are no zeroes, we say $o(e) = 0$.
\end{definition}

\begin{definition} The {\it Newton distance} of $g$, denoted by $d(g)$, is the infimum of all $x$ for which $(x,x) \in N(g)$. 
\end{definition}

The exponent $\eta$ of $(1.3)$ is often given by ${1 \over d(g)}$.  This holds whenever the line $y = x$ intersects $N(g)$ at a vertex, on the horizontal or vertical 
rays, or in the interior of a compact edge $e$ for which $o(e) \leq d(g)$ We refer to [G5] for proofs. It is not hard to show that one can only have
 $o(e) > d(g)$ if the slope of $e$
is either an integer or the reciprocal of an integer, and even in those cases one generically has $o(e) \leq d(g)$ in a sense described below after the statement of Theorem 1.5. Thus
in this way,  $\eta = {1 \over d(g)}$ is the generic situation. 

Whenever $\eta = {1 \over d(g)}$, one can show that one cannot go below the line $y = x - {\eta \over \eta + 1}$ in Theorem 1.3 by 
testing on rectangular boxes. Thus was explicitly worked out in the $\eta \leq {1 \over 2}$ case in Theorem 1.5 of [G3], where an analogous result was also shown
in the Sobolev space setting. Hence we have

\begin{theorem} Suppose that the line $y = x$ either intersects $N(S)$ at a vertex, on the horizontal or vertical ray, or in the interior of a 
compact edge $e$ with $o(e) \leq d(S)$. Then if 
$\phi$ is nonnegative and $\phi(0,0) > 0$, $T$ is not bounded 
from $L^p(\R^3)$ to $L^q(\R^3)$ for ${1 \over q} < {1 \over p} - {\eta \over \eta + 1}$. Consequently, if we also have $\eta \leq {2\eta' \over 1 + 2\eta'} $, then
 Theorem 1.3 is sharp up to endpoints; $T$ is not bounded from $L^p(\R^3)$ to $L^q(\R^3)$ outside the closed trapezoid of the theorem.
\end{theorem}

The condition on $N(S)$ in the above theorem is weaker than nondegeneracy in the sense of Varchenko [V], which requires $o(e) \leq 1$ for all
compact edges.
 Varchenko's condition is generic in the following sense. Suppose that $g(x,y)$ is a real analytic function on a neighborhood of 
the origin with $g(0,0) = 0$. Let $g_f(x,y) = \sum_e g_e(x,y)$ where the sum is taken over all compact dges
of $N(g)$. Then given a fixed Newton polygon $N$, the set of monomial coefficients of the $g_f(x,y)$  for which $g$  satisfying $N(g) = N$ is not nondegenerate is of measure zero in the appropriate $\R^n$. We refer to Chapter 6 of [AGuV] for more details on such matters. Since the Newton polygon conditions of 
Theorem 1.5 are weaker than nondegeneracy, one can view these conditions as generic in the same way. 
Thus in this sense generically one cannot go below the line $y = x - {\eta \over \eta + 1}$ of the trapezoid in Theorem 1.3. In view of the 
optimality of the other 
three sides, in this way generically Theorem 1.3 is best possible whenever we have $\eta \leq {2\eta' \over 1 + 2\eta'} $.  Hence the final statement of Theorem 1.5. As we will 
see in the next section, there is a similar sense in which generically one has $\eta = {2\eta' \over 1 + 2\eta'} $ and the intersection of the latter generic situation with the former 
can also be viewed as generic. In this sense Theorem 1.3 is generically optimal up to endpoints.

\section{The condition $\eta = {2\eta' \over 1 + 2\eta'} $}

\subsection{More about Newton polygons}

To help understand when $\eta = {2\eta' \over 1 + 2\eta'} $, and therefore we are in a setting  where
Theorem 1.1 and possibly 1.3 give results that are sharp up to endpoints, we again turn our attention to Newton polygons.
Let $S(x,y)$ be as in $(1.1)$, and we consider its 
Newton polygon $N(S)$. We denote the successive 
vertices of $N(S)$ by $(a_1,b_1),...,(a_n,b_n)$, where $b_{i+1} > b_i$ for each $i$. Let $e_i$ denote the edge of $N(S)$ connecting $(a_i,b_i)$ to
$(a_{i+1},b_{i+1})$ and we denote the slope of $e_i$ by $-{1 \over m_i}$, so that $m_i$ is decreasing in $i$. We stipulate that $m_0 = \infty$
and $m_n = 0$. We also make the following definition.

Let $H(x,y)$ denote the Hessian determinant of $S(x,y)$, and we suppose $H(0,0) = 0$ but $H$ is not identically zero. Suppose $(a_i,b_i)$ is a vertex of
$N(S)$ such that neither $a_i$ nor $b_i$ is zero. Suppose $l$ is a line in the $xy$ plane containing $(a_i,b_i)$ whose slope $-{1 \over m}$ satisfies 
$m_{i-1} > m > m_i$. Then $l$ intersects $N(S)$ in exactly one point, $(a_i,b_i)$. This situation can be described by saying that $a + mb$ is minimized
over all nonzero terms $S_{a,b} x^a y^b$ of the Taylor expansion of $S(x,y)$ in the single term $S_{a_i,b_i} x^{a_i} y^{b_i}$. Correspondingly, 
for such $m$, $a + mb$ is minimized over all nonzero terms of the Taylor expansion of $H(x,y)$ in the term given by the Hessian determinant of
$S_{a_i,b_i}x^{a_i}y^{b_i}$, namely $(S_{a_i,b_i})^2 a_ib_i(1 - a_i - b_i)x^{2a_i - 2}y^{2b_i - 2}$.

Thus if the line $y = x$ intersects $N(S)$ at the vertex $(d,d)$, the line $y = x$ intersects $N(H)$ at the vertex $(2d - 2, 2d - 2)$. By the earlier
discussion, this implies that  $\eta = {1 \over d}$ and $\eta' = {1 \over 2d - 2}$, in which case $\eta = {2\eta' \over 1 + 2\eta'} $. 
If the line $y = x$ 
intersects $N(S)$ in the interior of the vertical ray $x = d$ or horizontal ray $y = d$ that doesn't intersect the opposing coordinate axis, the line $y = x$ intersects $N(H)$ in the interior of the vertical ray $x  = 2d - 2$ or horizontal ray $y = 2d - 2$ respectively. Thus again $\eta = {1 \over d}$ and $\eta' = {1 \over 2d - 2}$, so that $\eta = {2\eta' \over 1 + 2\eta'} $. 

If the line $y = x$ intersects $N(S)$ in the interior of a bounded edge $e$ that doesn't intersect
either coordinate axis, then the line $y = x$ intersects $N(H)$ in the interior of a bounded edge $e'$ with the same slope. So for example if the slope of
this edge is not an integer or the reciprocal of an integer, once again $\eta = 
{1 \over d(S)}$ and $\eta' = {1 \over d(H)}$. The same holds whenever  $o(e) \leq d(S)$ and $o(e') \leq  d(H)$. In these situations, as before if
 the line  $y = x$ intersects $N(S)$ at $(d,d)$  the line $y = x$ intersects $N(H)$ at the vertex 
$(2d - 2, 2d - 2)$, so that again we have $\eta = {2\eta' \over 1 + 2\eta'} $.

Thus by the above discussion, we see that the only way we could have ${2\eta' \over 1 + 2\eta'} \neq \eta$ is if $N(S)$ has exactly one vertex, located
on a coordinate axis, or if the line $y = x$ intersects $N(S)$ in the interior of a bounded edge $e$ which either intersects a coordinate axis or satisfies 
one or both of the conditions $o(e) > d(S)$ and $o(e') > d(H)$. For given $N(S)$ and $N(H)$, these conditions are 
rare in the sense described at the end of section 1. When none 
of these exceptional
cases occur, both Theorem 1.1 and Theorem 1.3 will be sharp up to endpoints. Furthermore, as we will see in
the examples below, even if the edge $e$ intersects one or both coordinate axes one still often has sharpness in Theorems 1.1 and 1.3 in 
conjunction with an $\eta = {2\eta' \over 1 + 2\eta'} $ situation. The same is true if
$N(S)$ has exactly one vertex, located on a coordinate axis.

\subsection{Additional examples where $\eta \leq {2\eta' \over 1 + 2\eta'}$}

We consider the situation where the graph of $S(x,y)$ in $(1.1)$ is convex and of finite line type, meaning that every line tangent to the graph is 
not tangent the surface to infinite order. By a theorem of Schulz [Sc], after a linear coordinate change the Newton polygon $N(S)$ will have exactly two vertices, 
one of the form $(a,0)$ and the other of the form $(0,b)$ for even integers $2 \leq b \leq a$. Thus there is one bounded edge of $N(S)$, joining $(a,0)$ and
$(b,0)$. In this setting, Theorem 1.4 of 
the paper [ISaSe] gives $L^p$ to $L^q_s$ results. The $L^p$ to $L^p_s$ statement is as follows.

\begin{theorem}{\rm [ISaSe]} If the graph of $S$ is convex and of finite line type, there is a neighborhood $U$ of the origin
such that if $\phi(x,y)$ is supported in $U$, then $T$ is bounded from $L^p$ to $L^p_s$ if $({1 \over p}, s)$ is in the interior of
the polygon with vertices   $(0,0)$, $({1 \over a}, {2 \over a}), ({1 \over b}, {1 \over a} + {1 \over b}), (1 - {1 \over b}, {1 \over a} + {1 \over b}), 
(1 - {1 \over a}, {2 \over a}),$ and $(1,0)$. 
\end{theorem}

The statement for $L^p$ to $L^q$ boundedness in [ISaSe] can be derived from the above using a similar interpolation argument to one we used to get Theorems 1.3 and 1.4 from Theorems 1.1 and 1.2. Note that the edges of the polygon in the above theorem have slopes $2,1,0,-1,$ and $2$ respectively,
with the edges of slope $1$ and $-1$ disappearing if $a = b$.
The above result can be stronger than that of Theorem 1.1 if $\eta > {2\eta' \over 1 + 2\eta'}$, such as in the case where $S(x,y) = x^a + y^b$
for distinct even integers $a$ and $b$. But there are many situations where $\eta = {2\eta' \over 1 + 2\eta'}$, in which case Theorems 1.1 and 1.3 are 
sharp up to endpoints and will give stronger results. We next describe a class of examples that illustrates this phenomenon.

\noindent {\bf Example 1.} Suppose $S(x,y) = y^4 + ay^2x^4 + x^8 + f(x,y)$, where $a$ is any real number and $f(x,y)$ is a real analytic function
all of whose nonzero Taylor expansion terms $f_{a,b}x^ay^b$ are such that $(a,b)$ is above the line joining $(8,0)$ to $(0,4)$. Thus $N(S)$ has 
two vertices, $(0,4)$ and $(8,0)$, and a single edge joining them. Then a direct computation reveals that $H(x,y) = 144ax^2y^4 + (672 - 40a^2)x^6y^2
+ 112a x^{10} + g(x,y)$, where every nonzero term $g_{a,b}x^ay^b$ of $g$'s Taylor expansion lies above the line joining $(10,0)$ to $(0,5)$.

The Newton distances $d(S)$ and $d(H)$ are computed readily, with $d(S) = {8 \over 3}$ and $d(H) = {10 \over 3}$ if $a \neq 0$. The line $y = x$ intersects $N(S)$
in the interior of a compact edge $e$ with $S_e(x,y) =   y^4 + ay^2x^4 + x^8$, and when $a \neq 0$, the line $y = x$ intersects $N(H)$
in the interior of a compact edge $e'$ with $H_{e'}(x,y) =  144ax^2y^4 + (672 - 40a^2)x^6y^2 + 112a x^{10} $. Since both $S_e(x,y)$ and $H_{e'}(x,y)$
are functions of $y^2$ for fixed $x$, the maximum order of any zero of $S_e(1,y)$ or $H_{e'}(1,y)$ for $y \neq 0$ is two, which is less than either 
Newton distance. Thus here $\eta = {1 \over d(S)} = {3 \over 8}$ and $\eta' = {1 \over d(H)} = {3 \over 10}$, and we have ${2\eta' \over 1 + 2\eta'} 
= \eta$. Hence whenever $a \neq 0$, we  are in the situation where Theorems 1.1 and 1.3 give estimates that are sharp up to endpoints.

We next consider when this class of examples corresponds to surfaces that are convex and of finite line type. For simplicity we assume $f(x,y)  = 0$
so that $S(x,y) = y^4 + ay^2x^4 + x^8$ and $H(x,y) = 144ax^2y^4 + (672 - 40a^2)x^6y^2 + 112a x^{10}$.
 A necessary condition is that the Hessian 
determinant is never negative. This rules out $a$ being negative as in this case the Hessian determinant is negative on the $x$ axis. If $a > 0$ and the discriminant of
$b(z) = 144az^2 + (672 - 40a^2)z + 112a $ is negative, then $b(z)$ has complex conjugate roots and thus $b(z)$ is a positive function. The same will 
then be true for $H(x,y) =  144ax^2y^4 + (672 - 40a^2)x^6y^2 + 112a x^{10}$ except on the $y$ axis. Thus $S(x,y)$ must be either strictly convex or strictly 
concave outside the $y$ axis but since $S(x,0)$ is convex in $x$ for example, we conclude $S(x,y)$ is strictly convex outside the $y$ axis.
As for the $y$ axis, one can directly verify that at each $(0,y_0)$ and each direction $v$, the second directional derivative in the $v$ direction is nonnegative, and some higher directional derivative is nonzero. We conclude that $S(x,y)$ is convex and of finite line type everywhere if the discriminant of $b(z)$ is 
negative. Similarly, if the discriminant of $b(z)$ is positive, but the roots of $b(z)$ are negative, corresponding to the case where $672 - 40a^2 > 0$, then $H(1,y) = 144ay^4 + (672 - 40a^2)y^2 + 112a$ has purely imaginary roots and by similar considerations to the above, $S(x,y)$ will be convex and of finite line type.

In summary, $S(x,y)$ will be convex and of finite line type if $a > 0$ and either the discriminant of $b(z)$ is negative or the discriminant of $b(z)$ is 
positive but $672 - 40a^2 > 0$. This discriminant is calculated to be $64(25a^4 - 1848 a^2 + 7056)$, which is positive from $0$ to $2.0097...$, negative
from $2.0097...$ to $8.3595...$, and positive beyond $8.3595....$. In the first interval one has $672 - 40a^2 > 0$, so we conclude that $S(x,y)$ is convex
and of finite line type whenever $0 < a < 8.3595...$ In these cases Theorems 1.1 and 1.3 give optimal boundedness domains up to endpoints, and these
will be larger than that of the $a = 0$ case, which is the domain given by Theorem 1.4 of [ISaSe]. 

The phenomenon of the above example occurs in various other families, convex and nonconvex. We give the following example.

\noindent {\bf Example 2.} Suppose $S(x,y) = y^3 + ay^2x + byx^2 + x^3 + f(x,y)$ for some real $a$ and $b$, where $f(x,y)$ has a zero of order
$4$ or greater at the origin. Note that any function with a zero of order 3 at the origin can be converted into this form after an appropriate linear transformation. Then $N(S)$ has vertices at $(0,3)$ and $(3,0)$,
and a single edge $e$ that connects them. The Newton distance $d(S)$ is then ${3 \over 2}$. 
The function $S_e(1,y) = y^3 + ay^2 + by + 1$ has a double root $r$ if and only if $(a,b) = (-2r + {1 \over r^2}, r^2 - {2 \over r})$, so the set of $(a,b)$
for which $S_e(1,y)$ has a double root has measure zero. If $(a,b)$ is not in this set then $o(e) = 0$ or $1$, so that $o(e) < d(S)$ and as a result 
$\eta = {1 \over d} = {2 \over 3}$.

The Hessian of $S(x,y)$ can be computed to be $H(x,y) = (12b - 4a^2)y^2 + (36 - 4ab)xy + (12a - 4b^2)x^2 + g(x,y)$ for some $g(x,y)$ with a zero
of order $3$ or greater at the origin.
Suppose $12b - 4a^2, 12a - 4b^2 \neq 0$. Then $N(H)$ will have one edge $e'$, connecting the two vertices $(2,0)$ and $(0,2)$.
Note that $d(H) = 1$ and that
$H_{e'}(1,y) = (12b - 4a^2)y^2 + (36 - 4ab)y + (12a - 4b^2) $ will only have a double root if the discriminant $(36 - 4ab)^2 - 4(12b - 4a^2)(12a - 4b^2)$ is zero. 
This corresponds to $(a,b)$ belonging to a curve in $\R^2$. Thus for $(a,b)$ outside a set of measure zero, $o(e') \leq 1 = d(H)$. For these situations
$\eta' = 1$. 

Combining the above, we see that for $(a,b)$ outside a set of measure zero, we have $\eta = {2\eta' \over 1 + 2\eta'}$ and Theorem 1.1 is sharp up to 
endpoints. Since we have $o(e) \leq d(S)$ and $o(e') \leq d(H)$ in these situations, we also have that Theorem 1.3 is sharp up to endpoints.

We next give a class of examples corresponding to the case where $N(S)$ has exactly one vertex, which lies on a coordinate axis.

\noindent {\bf Example 3.} Let $S(x,y) = y^n + axy^n $ for some $n \geq 2$, $a \neq 0$. 
 Then $N(S)$ has one vertex, at $(0,n)$. In this case the
Hessian $H(x,y)$ is $-a^2n^2y^{2n - 2}$. Thus $d(S) = n$ and $d(H) = 2n - 2$. In both cases the line $y = x$ intersects the 
Newton polygon in the interior of the horizontal ray, so we have $\eta = {1 \over d(S)} = {1 \over n}$ and $\eta' = {1 \over d(H)} = {1 \over 2n - 2}$ respectively. Thus once again we have 
$\eta = {2\eta' \over 1 + 2\eta'}$ and the boundedness regions provided by Theorem 1.1 or 1.3 are optimal up to endpoints. 

Lastly, we give a class of examples that illustrates that one can have $\eta < {2\eta' \over 1 + 2\eta'}$. This class of functions also shows you can have
sharpness in Theorem 1.1 without having $o(e) \leq d(S)$ for the edge $e$ of $N(S)$ intersecting the line $y = x$.

\noindent{\bf Example 4.} Let $S(x,y) = (y - x^m)^n$ for some integers $m, n \geq 2$. Then by changing variables from $(x,y)$ to $(x, y - x^m)$, we see
that $\eta$ here is that of $y^n$, namely ${1 \over n}$. On the other hand, the Hessian determinant  of $S(x,y)$ is $H(x,y) = -(m-1)m(n-1)n^2 x^{m-2}(y - x^m)^{2n-3}$. 
By the same variable change, we see that $\eta'$ here is that of $x^{m-2}y^{2n - 3}$, namely $\min({1 \over m - 2}, {1 \over 2n - 3})$. Thus
${2\eta' \over 1 + 2\eta'} = \min({2 \over m}, {2 \over 2n - 1})$, which is at least $\eta$ whenever $m \leq 2n$.
Thus whenever $m \leq 2n$, Theorem 1.1 is sharp up to endpoints. 

On the other hand, the Newton polygon of $S(x,y)$ has an edge $e$ connecting $(mn,0)$ to $(0,n)$, so that the Newton distance $d(S)$
 is ${mn \over m + 1}$. Thus $o(e) = n$ is always greater than $d(S)$. Hence our condition giving sharpness up to endpoints in Theorem 1.3 does not hold here. 

\section{Resolution of singularities and some consequences}

\subsection{The original resolution of singularities theorem}

We now make use of resolution of singularities results theorems from [G4], which were also used in [G2].  Suppose we have real analytic functions $f_1(x,y),...,f_k(x,y)$
 on a neighborhood of the origin, none identically zero, with
$f_j(0,0) = 0$ for each $j$.  Denote the Taylor expansion of  $f_j(x,y)$ by $\sum_{\alpha,\beta} f_{\alpha\beta}^j x^{\alpha}y^{\beta}$ and
denote by $o_j$ the order of the zero $f_j(x,y)$ at the origin. After rotating coordinates if necessary, we assume $f_{0, o_j}^j \neq 0$ for each $j$.

Divide the
$xy$ plane into eight triangles by slicing the plane using  the $x$ and $y$ axes  and two lines through the origin, one of the form $y = mx$ for some $m > 0$ and one of the form $y = mx$ for some $m < 0$.  These two lines should not be lines on which
which any  function $\sum_{\alpha + \beta = o_j} f_{\alpha\beta}^jx^{\alpha}y^{\beta}$ vanishes, except at the
origin. After reflecting about the $x$ and/or $y$ axes and/or the line $y = x$ if necessary, each of the triangles becomes of the form $T_b = \{(x,y) \in \R^2: x > 0,\,0 < y < bx\}$ (modulo an inconsequential boundary set of measure zero). Then Theorems 2.1 and 2.2 of [G4] give the following.

\begin{theorem}  Let  $T_b = \{(x,y) \in \R^2: x > 0,\,0 < y < bx\}$ be as above. Abusing notation slightly, use the notation $f_j(x,y)$ to denote the reflected function $f_j(\pm x,\pm y)$ or $f_j(\pm y, \pm x)$ corresponding to $T_b$.
 Then there is a $a > 0$ and a positive integer $N$ such that
if $F_{a,b} = \{(x,y) \in \R^2: 0 \leq x\leq a, \,0 \leq y \leq bx\}$, then one can write $F_{a,b} = \cup_{i=1}^n cl(D_i)$, such that for to each $i$ there is a $k_i(x) = l_i x^{s_i} + ...$ with $k_i(x^N)$ real analytic and $s_i \geq 1$ such that after a coordinate change of the form $\eta_i(x,y) = (x, \pm y + k_i(x))$, the set $D_i$ becomes a set $D_i'$ in the upper right quadrant on which each function $f_j \circ \eta_i(x,y)$ approximately becomes a monomial 
$d_{ij} x^{\alpha_{ij}}y^{\beta_{ij}}$, $\alpha_{ij}$ a nonnegative rational number and $\beta_{ij}$ a nonnegative integer in the following sense.

\begin{enumerate}
\item  $D_i' = \{(x,y): 0 < x < a, \, g_i(x) < y < G_i(x)\}$, where $g_i(x^N)$ and $G_i(x^N)$ are
real analytic. If we expand $G_i(x) =  H_i x^{M_i} + ...$, then $M_i \geq 1$ and $H_i > 0$. 
\item The function $g_i(x)$ is either identically zero or $g_i(x)$ 
can be expanded as $h_ix^{m_i} + ...$ where $h_i > 0$ and $m_i > M_i$. The $D_i'$ can
be constructed such that such that for any predetermined $\kappa > 0$ there is a $d_{ij} \neq 0$ such that on $D_i'$, for all $0 \leq l \leq \alpha_{ij}$ and all $0 \leq  m \leq \beta_{ij}$ one has
\[|\partial_x^l\partial_y^m(f_j \circ \eta_i)(x,y) -  \alpha_{ij}(\alpha_{ij}- 1)...(\alpha_{ij}- l + 1)\beta_{ij}(\beta_{ij} - 1)...(\beta_{ij} - m + 1)
d_{ij}x^{\alpha_{ij}- l}y^{\beta_{ij} - m}| \]
\[\leq \kappa |d_{ij}| x^{\alpha_{ij}- l}y^{\beta_{ij} - m} \tag{3.1}\]
\end{enumerate}
\end{theorem}
\noindent As stated in Corollary 2.3 of [G4], the proofs of Theorems 2.1 and 2.2 of [G4] imply the following corollary.

\begin {corollary} For any given $K$, however large, the $D_i'$ can be constructed 
so that there is a constant $C_K$ so that on $D_i'$ one has $|\partial_x^{a}\partial_y^{b} (f_j \circ\eta_i) (x,y)| \leq C_K x^{-a}y^{-b}
 |(f_j\circ \eta_i)(x,y)|$ for all $a, b < K$ and all $j$.
\end{corollary}

For our purposes, we need a slight extension of the above corollary. Namely we need it to hold for arbitrarily many $y$ derivatives. This also follows from the proofs of Theorem 2.1 and 
2.2 of [G4].

\begin {corollary} The $D_i'$ can be constructed 
so that for each $b$ there is a constant $C_b$ so that on $D_i'$ one has $|\partial_y^{b} (f_j \circ\eta_i) (x,y)| \leq C_b y^{-b}
 |(f_j\circ \eta_i)(x,y)|$ for all $b$ and all $j$.
\end{corollary}

The basic idea behind why Corollary 3.1.2 holds is as follows. When $f_j \circ \eta_i \sim x^{\alpha_{ij}}y^{\beta_{ij}}$ on $D_i'$, the Newton polygon $N(f_j \circ \eta_i)$ has a vertex at
$(\alpha_{ij}, \beta_{ij})$ and on $D_i'$, $x^{\alpha_{ij}}y^{\beta_{ij}}$ dominates any other nonzero term $c_{\alpha,\beta} x^{\alpha}y^{\beta}$ of the Taylor expansion of $f_j \circ \eta_i $ due to the fact that 
 $(\alpha,\beta)$ is  contained in this Newton polygon.
 Therefore for any $b$, $x^{\alpha_{ij}}y^{\beta_{ij}}$ will, up to some constant depending on $b$, dominate any other nonzero term of the Taylor expansion of 
$y^b \partial_y^b (f_j \circ \eta_i)$ since the Newton polygon of $y^b \partial_y^b (f_j \circ \eta_i)$ is a subset of that of $f_j \circ \eta_i$.

The resolution of singularities theorem above is compatible with smooth functions in the following sense, as follows directly from the constructions 
of Theorems 2.1 and 2.2 of [G4].

\begin{theorem} Suppose we are in the setting of Theorem 3.1, and suppose $\psi(x,y)$ is a smooth bump function on $\R^2$. Then on $\R^2 - 
\{(0,0)\}$ we can write
$\psi(x,y) =\sum_{ijk} \psi_{ijk}(x,y)$ such that the following hold for some constant $C$ independent of $\psi$.
\begin{enumerate}
\item Each $\psi_{ijk} \circ \eta_i(x,y)$ is a smooth function supported on $[C^{-1}2^{-j}, C 2^{-j}] \times [C^{-1}2^{-k}, C 2^{-k}]$.
\item For each nonnegative integer $a$ and $b$ there is a constant $D_{a,b,\psi}$ such that for  each $(i,j,k)$ one has
\[ |\partial_x^a\partial_y^b (\psi_{ijk} \circ \eta_i(x,y))| \leq D_{a,b,\psi} 2^{aj + bk}\tag{3.2}\]
\item The inequality $(3.1)$ as well as Corollaries 3.1.1 and 3.1.2 hold on the support of  $\psi_{ijk} \circ \eta_i(x,y)$
\end{enumerate}
\end{theorem} 

\subsection{Adjusting the algorithm}

We will need a variant of the above theorems, which in the context of Theorem 1.1
 we will apply to $S(x,y)$, $\partial_y^2 S(x,y)$, $\partial_y(S(x,y) - S(0,y))$, and the 
Hessian determinant of $S(x,y)$, which we are calling $H(x,y)$. The purpose of resolving the singularities of $\partial_y(S(x,y) - S(0,y))$ is to ensure
$\partial_{xy} S(x,y)$ is monomialized in the new coordinates in the sense of part 2 of Theorem 3.1. We don't resolve the singularities of $\partial_{xy} S(x,y)$ itself since this function
does not behave as well as  $\partial_y(S(x,y) - S(0,y))$ under the types of coordinate changes we are doing in this paper, and resolving the singularities of  $\partial_y(S(x,y) - S(0,y))$ 
will ensure the singularities of $\partial_{xy} S(x,y)$ are resolved. 

The variant proceeds as follows.
After doing the rotations and reflections preceding Theorem 3.1, we do some further linear maps. Namely, if $o$ denotes the order of the zero of 
(the rotated/reflected) $S(x,y)$
at $(0,0)$, let $S_0(x,y) = \sum_{\alpha + \beta = o} S_{\alpha, \beta}x^{\alpha}y^{\beta}$, the sum of the terms of $S$'s Taylor expansion of lowest
order. We divide the domain triangle $\{(x,y) \in \R^2: 0 < x < a,\,0 < y < bx\}$ into smaller triangles $U_l$, then do linear maps of the
form $L(x,y) = (x, y + m_lx)$, $m_l \in \R$,  to place the lower boundaries of the $U_l$ on the $x$ axis. We do this in such a way such that if $U_l'$ denotes the domain $U_l$ in the
new coordinates, if $U_l'$ is written as $\{(x,y) \in \R^2: 0 < x < a,\,0 < y < c_lx\}$, then the function $S_0(x, y + m_lx)$ is such that
$S_0(1, y + m_l)$ either has no zeroes on $[0,c_l]$, or has a single zero, located at $y = 0$.

Let $R_l(x,y)$ denote $S(x, y + m_lx)$, viewed as a function on $U_l'$. We next look at the Newton polygon $N(R_l)$. Because of the form of the
domains $U_l'$, we will only be interested in the edges of $N(R_l)$ of slope $-{1 \over m}$ for some $m \geq 1$. We first have the following lemma.

\begin{lemma} Suppose $e$ is an edge of $N(R_l)$ of slope $-{1 \over m}$ for $m \geq 1$,  and  suppose $y_0 \geq 0$ is such that the 
associated polynomial $(R_l)_e(1,y)$ (see Definition 1.2) satisfies $ (R_l)_e(1,y_0) \neq 0$ and $\partial_y (R_l)_e(1,y_0) = 0$. Then there is a small wedge $W_{y_0} = \{(x,y) \in \R^2: 0 < x < a,\,c x^m < y < c' x^m \}$, where $c < y_0 < c'$,  on which for some $\alpha > 0$ 
(depending on $y$), we have $|R_l(x,y)| \sim x^{\alpha}$ and $|\partial_x R_l(x,y)| \sim x^{\alpha - 1}$. If $y_0 = 0$ we also have 
$|\partial_{xx}R_l(x,y)| \sim x^{\alpha - 2}$.
\end{lemma}
\begin{proof}
First note that up to small error terms,
$R_l(x,y)$ equals the mixed homogeneous polynomial $(R_l)_e(x,y)$ on such a $W_{y_0}$. Since $ (R_l)_e(1,y_0) \neq 0$, this means that
$|R_l(x,y)| \sim |(R_l)_e(x,y)| \sim x^{\alpha}$ for some $\alpha$ on such a $W_{y_0}$.  Since $ (R_l)_e(1,y_0) \neq 0$, the mixed homogeneity
combined with the fact that the $y$ derivative vanishes at $y_0$ ensures that the $x$ derivative $|\partial_x (R_l)_e(x,y)|$ will not be vanishing
on the curve $y = y_0 x^m$. Thus for a small enough wedge $W_{y_0}$ we have $|\partial_x (R_l)_e(x,y)| 
\sim x^{\alpha - 1}$ on $W_{y_0}$, and therefore $|\partial_x R_l(x,y)| \sim x^{\alpha - 1}$ as well close enough to the origin. In the case that $y_0 = 0$, one can argue directly; since  $(R_l)_e(x,0) \neq 0$ one has $\partial_x(R_l)_e(x,0) \sim x^{\alpha - 1}$ and $\partial_{xx} (R_l)_e(x,0) \sim x^{\alpha - 2}$. Thus if $W_{y_0}$ is 
sufficiently small, on $W_{y_0}$ one similarly has $|\partial_xR_l(x,y)| \sim x^{\alpha - 1}$ and $|\partial_{xx}R_l(x,y)| \sim x^{\alpha - 2}$.
\end{proof}

In what follows, we will always assume that for $y_0 \neq 0$, the wedge $W_{y_0}$ is chosen small enough so that $0 \notin [c,c']$.

The variant of Theorems 3.1 and 3.2 we will need is as follows. For any wedge $W_{y_0}$ as in Lemma 3.3 we reverse the roles of the
$x$ and $y$ variables and consider the function $Q_l(x,y) = R_l(y,x)$ on the reflected set $W_{y_0}'$. If $y_0 \neq 0$, the 
set $W_{y_0}'$ is a wedge of the form $ \{(x,y) \in \R^2: 0 < y < a,\,c' x^{1 \over m} < y < c x^{1 \over m} \}$ on which we have 
$|Q_l(x,y)| \sim x^{{\alpha \over m}}$ and $|\partial_y Q_l(x,y)| \sim  x^{{\alpha \over m} - {1 \over m}}$. If $y_0 =  0$, then 
$W_{y_0}'$ is a wedge of the form $ \{(x,y) \in \R^2: 0 < y < a, y > c x^{1 \over m} \}$ on which we have $|Q_l(x,y)| \sim y^{\alpha}$, 
$|\partial_y Q_l(x,y)| \sim  y^{\alpha -1 }$, and $|\partial_{yy} Q_l(x,y)| \sim  y^{\alpha -2 }$. Although these forms are different from those in Theorems 3.1 and 3.2, one can still apply the resolution of singularities algorithm on such wedges. So we apply the algorithm to $Q_l(x,y)$, $\partial_y^2 Q_l(x,y)$,
 $\partial_y(Q_l(x,y) - Q_l(0,y))$, and the Hessian determinant of $Q_l(x,y)$. Then Theorem 3.1 and Corollaries 3.1.1 and 3.1.2 will all still hold.

On the portions of the domains $U_l'$ that are not part of any $W_{y_0}$, we simply apply the original resolution of singularities algorithm to
$R_l(x,y)$, $\partial_y^2 R_l(x,y)$,
 $\partial_y(R_l(x,y) - R_l(0,y))$, and the Hessian determinant of $R_l(x,y)$, and 
Theorem 3.1 and Corollaries 3.1.1 and 3.1.2 will hold as usual.

The reason we make the above modifications of the resolution process is that we want to be in one of the two situations in Theorem 3.4 on every $D_i'$. If we had not done the above
reversals of the roles of the $x$ and $y$ variables, there could be situations where $ (R_l)_e(1,y_0) \neq 0$ and $\partial_y (R_l)_e(1,y_0) = 0$ for which neither part of the lemma is
 satisfied on $D_i'$  coming from the associated $W_{y_0}$.

In what follows we let $\eta_i$ denote any of the coordinate change maps occurring in the resolution of singularities process above, whether or not it
derives from a $W_{y_0}$. In other words, the original (rotated/reflected) $S(x,y)$ on the triangle $\{(x,y) \in \R^2: 0 < x < a,\,0 < y < bx\}$
 becomes $S \circ \eta_i(x,y)$ in the final coordinates. We similarly let $D_i'$ denote any of the final domains occurring, so that in this notation
  Theorem 3.1 and Corollaries 3.1.1 and 3.1.2 hold.
As for Theorem 3.2, since the subdivisions forming the $W_{y_0}$ are done at the initial stages
of the resolution of singularities procedure, they do not interfere with the result and Theorem 3.2 still holds in the current setting.

\subsection{Some useful consequences}

In the case that $S \circ \eta_i(x,y) \sim x^{\alpha_i}$ on $D_i'$ for some $\alpha_i$ (i.e. $\beta_{ij} = 0$ for $S \circ \eta_i$ in Theorem 3.1), we will need some 
second derivative estimates on $S \circ \eta_i(x,y)$ for our future results. These will be given by the next two theorems. The first is as follows.

\begin{theorem} Let $S_i(x,y)$ denote $S \circ \eta_i(x,y)$, and suppose $S_i(x,y) \sim x^{\alpha_i}$ on $D_i'$. Let $o$
 denote the order of the zero of $S(x,y)$ at the origin. Then for each
$i$, at least one of the following two situations holds.
\begin{enumerate}

\item There exists a constant $C$ and $g_i, s_i  > 0$ such that for each $j$ and $k$, one has $|{\partial^2 S_i \over\partial x\partial y} (x,y)| > C^{-1}x^{s_i - g_i - 1}$ on the support of $\psi_{ijk}$, and
such that $|S_i(x,y)| < Cx^{s_i}$ on a product of intervals $R_1 \times R_2  \subset [2^{-j}, 2^{-j+1}]  \times [C^{-1} 2^{-jg_i}, C2^{-jg_i}]$,
where $|R_1| > C^{-1} 2^{-j}$, $|R_2| > C^{-1}2^{-jg_i}$, and $k \geq jg_i$.
\item The lowest edge of the Newton polygon $N(S_i)$ joins $(\alpha_i, 0)$ to $(\alpha_i',\beta_i')$ for some $\beta_i' \geq 2$. If the slope 
$-{1 \over m}$ of this edge satisfies $m < 2$ then $\beta_i' \leq {o \over 2}$. If $m \geq 2$ then $\beta_i' \leq o$. 
\end{enumerate}
\end{theorem}

\begin{proof} 
\
\

We will be making use of certain aspects of the proof of Theorem 3.1 (Theorems 2.1 and 2.2 of [G4]), and to have the fullest understanding of this proof it would be helpful to be 
familiar with the proofs of those theorems.

We first go back to the domains $U_l'$ defined subsequent to the statement of Theorem 3.2 For a given such $U_l'$, we examine the Newton
polygon of the function we denoted by $R_l(x,y)$, which was a linear shift of $S(x,y)$. We let $(a_1,b_1),....,(a_n,b_n)$ denote the vertices of $N(R_l)$,
where $a_p > a_{p+1}$ for each $p$, and we let $e_p$ denote the edge of $N(R_l)$ joining $(a_p,b_p)$ to $(a_{p+1},b_{p+1})$. We write the slope of $e_p$
as $-{1 \over m_p}$. For large enough $N$, if $p\neq 1$ or $n$, we will have $R_l(x,y) \sim x^{a_p}y^{b_p}$ on any set $A_p = 
\{(x,y) \in U_l' : Nx^{m_{p-1}} < y < 
{1 \over N}x^{m_p}\}$. The same will be true for $p = 1$ on the set $A_1 = \{(x,y) \in U_l':  y < {1 \over N}x^{m_1}\}$, and for $p = n$ on the set
$A_n = \{(x,y) \in U_l': Nx^{m_{n-1}} < y \}$. 

If $S_i(x,y) \sim x^{\alpha_i}$ on $D_i'$ and $D_i'$ derives from one of the above sets $A_p$ for $b_p \geq 1$, then for the associated vertex $(a_p,b_p)$, 
$\alpha_i$ will equal $a_p + mb_p$, where $m$ is such that $D_i'$ is derived from the portion of $A_p$ where $y \sim x^m$. In this case 
we will be in case 1 of the theorem; we will have that $\partial_yR_l(x,y) \sim x^{a_p + mb_p - m}$, which implies that
$\partial_y S_i(x,y) \sim x^{a_p + mb_p - m}$ and therefore $\partial_{xy} S_i(x,y) \sim x^{a_p + mb_p - m - 1}$ on $A_i$, which translates into case 1
holding with $g_i = m$ and $s_i = a_p + mb_p$. Here $R_1 \times R_2$ is a dyadic rectangle in the coordinates of $U_l'$, transformed into
the coordinates of $D_i'$.

The remaining situations are where $D_i'$ derives from a set that is either of the form $\{(x,y) \in U_l':  {1 \over N} x^{m_p} < y < N x^{m_p}\}$, 
where the lower vertex $(a_p,b_p)$ of the edge $e_p$ satisfies $b_p \geq 1$, or is of the form 
$\{(x,y) \in U_l':  0 <  y < N x^{m_p}\}$, where the lower vertex of the edge $e_p$ is of the form $(a_p,0)$.

 The constructions of the resolution of singularities algorithms are such that
there are a few possibilities. The first is that $D_i'$ derives from a set of the form $B = \{(x,y) \in U_l':  c x^{m_p} < y < c' x^{m_p}\}, 0 \leq c < c'$, 
for which 
$(R_l)_{e_p}(1,y)$ and $\partial_y (R_l)_{e_p}(1,y)$ are both nonzero on $[c,c']$. In this case we have $(R_l)_{e_p}(x,y) \sim x^{a_p+ m_pb_p}$ and 
$\partial_y (R_l)_{e_p}(x,y) \sim x^{a_p + m_pb_p - b_p}$ on $B$, so since the error terms are negligible on $B$ we also have 
$R_l (x,y) \sim x^{a_p + m_pb_p}$, $\partial_y R_l(x,y) \sim x^{a_p + m_pb_p - b_p}$, 
on $B$. The same will hold true with $R_l$ replaced by $S_i$, which then implies that $\partial_{xy}S_i(x,y) \sim x^{a_p + m_pb_p - b_p - 1}$.
This places us into case 1 of this theorem, where $s_i = a_p + m_pb_p$,
$g_i = m_p$ and again $R_1 \times R_2$ is a dyadic rectangle in $U_l'$, transformed into the coordinates of $D_i'$.

The next possibility we consider for the $\{(x,y) \in U_l':  {1 \over N} x^{m_p} < y < N x^{m_p}\}$, $b_p \geq 1$ or 
$\{(x,y) \in U_l':  0 <  y < N x^{m_1}\}$, $b_1 = 0$ situation is when $D_i'$ derives from a 
set of the form $\{(x,y) \in U_l':  c x^{m_p} < y < c' x^{m_p}\}$, $0 \leq c < c'$,  for which  $(R_l)_{e_p}(1,y)$  is nonzero on $[c,c']$ but $\partial_y (R_l)_{e_p}(1,y)$
has a zero in $[c,c']$. In this case, our constructions are such that we may assume that $\{(x,y) \in U_l':  c x^{m_p} < y < c' x^{m_p}\}$ is a subset of one of the
$W_{y_0}$ of Lemma 3.3, that  $ y_0$ is the only zero of $(R_l)_{e_p}(1,y)$  in $[c,c']$, and that either $y_0 = 0$ or $0 < c < c'$.

 If $y_0 = 0$, then switching 
 the $x$ and $y$ variables either puts us into a situation where either no further resolution of singularities is needed and $|\partial_{xx}R_l(x,y)| \sim x^{\alpha - 2}$ from  becomes
$|\partial_{yy}\tilde{R}_l(x,y)|  = |\partial_{yy}S_i(x,y)| \sim y^{\alpha - 2}$ in the new coordinates, meaning we are no longer in the $\beta_{ij} = 0$ situation of this theorem, or
if some further resolution of singularities is needed, we can assume we are restricted to a domain which is a subset of a wedge $y \sim x^m$ for some $x$.
 In this case $|\tilde{R}_l(x,y)| \sim x^{m\alpha}$
and $|\partial_{y}\tilde{R}_l(x,y)| \sim x^{m\alpha - m }$,  so that $|\partial_{y}S_i(x,y)| \sim x^{m\alpha - m}$ and $|\partial_{xy}S_i(x,y)| \sim x^{m\alpha - m - 1}$. Therefore we are in the setting of part 1 of this theorem, with  the $R_1 \times R_2$ 
coming from dyadic rectangles in the reflected regions from which $D_i'$ derives. 
 
  If $y_0 \neq 0$, after reversing the roles of the $x$ and $y$ variables, 
 the statement from Lemma 3.3 that 
 $|\partial_{x}R_l(x,y)| \sim x^{a_p + m_pb_p  - 1}$ becomes the statement 
  $|\partial_{y}\tilde{R}_l(x,y)| \sim x^{ {a_p \over m_p} + b_p -  {1 \over m_p}}$, which implies that $|\partial_y S_i(x,y)| \sim 
  x^{ {a_p \over m_p} + b_p -  {1 \over m_p}}$ and thus  $|\partial_{xy}S_i(x,y)| \sim x^{ {a_p \over m_p} + b_p 
 - 1 - {1 \over m_p}}$, so we are again in the setting of part 1 of this theorem, similar to the last paragraph.
 
Next, we consider the $\{(x,y) \in U_l':  {1 \over N} x^{m_p} < y < N x^{m_p}\}$, $b_p \geq 1$ or 
$\{(x,y) \in U_l':  0 <  y < N x^{m_1}\}$, $b_1 = 0$ situation when $(R_l)_{e_p}(1,y)$  has a 
zero $r \neq 0$ in $[c,c']$ for which $\partial_y (R_l)_{e_p}(1,r) \neq 0$. Then considerations as in four paragraphs ago show that 
$\partial_y S_i(x,y) \sim x^{a_p + m_pb_p - b_p}$, and  $\partial_{xy}S_i(x,y) \sim x^{a_p + m_pb_p - b_p - 1}$, while $S_i(x,y) \sim x^{\alpha_i}$ for
some $\alpha_i \geq a_p + m_pb_p$. Thus we are again in case 1 with $g_i = m_p$, $s_i = a_p + m_pb_p$, and $R_1 \times R_2$ a dyadic rectangle
 in $U_l'$, transformed into the coordinates of $D_i'$.

Lastly, we consider the $\{(x,y) \in U_l':  {1 \over N} x^{m_p} < y < N x^{m_p}\}$, $b_p \geq 1$ or 
$\{(x,y) \in U_l':  0 <  y < N x^{m_1}\}$, $b_1 = 0$ situation when $(R_l)_{e_p}(1,y)$  has a 
zero $r \neq 0$ in $[c,c']$ of order two or greater. Again we may assume that it is the unique zero of $(R_l)_{e_p}(1,y)$ in $[c,c']$ and $c > 0$. Our constructions are such that $m_p \geq 1$ for any domain $U_l'$ where the roles of the $x$ and $y$ variables have not been switched, such as the present situation. Furthermore, due to the linear coordinate shifts described above Lemma 3.3,
if  $(R_l)_{e_p}(x,y)$ has zeroes in such a $U_l'$ then in fact $m_p > 1$. So $m_p > 1$ here. 

The maximum order of a zero of $(R_l)_{e_p}(1,y)$ is the $y$-coordinate of the intersection of the line containing $e_p$ with the $y$-axis, which is at most $o$ since in the
beginning of the resolution of singularities process, we did a rotation to ensure the Newton polygon of $S$ had a vertex at $(0,o)$. Furthermore, 
if $1 < m_p < 2$, since $m_p$ is not a multiple of an integer, the maximum possible
order of a zero $r \neq 0$ of $(R_l)_{e_p}(1,y)$ is ${o \over 2}$ since in this case $(R_l)_{e_p}(1,y)$ must skip at least every other degree. 
We conclude that the order of the zero $r$ of $(R_l)_{e_p}(1,y)$ is at most $o$, and is at most ${o \over 2}$ when
$m_p < 2$.
 
The resolution of singularities process is such that if $D_i'$ corresponds to this higher order zero situation, then in the final coordinates, $N(S_i)$ has a
 vertex at a point $(c_i,d_i)$ where $d_i$ is the order of the zero of $(R_l)_{e_p}(1,y)$ at $r$. Furthermore, since $S_i \sim x^{\alpha_i}$ here, the lowest edge of $N(S_i)$ connects 
$(0,\alpha_i)$ to some $(\alpha_i',\beta_i')$, where $\beta_i' \leq d_i$. If $\beta_i' = 1$, then we are in case 1 of this theorem;  in this situation we have that $\partial_{xy} S_i \sim x^{\alpha_i' - m -  1}$ on $D_i'$ where $m  = \alpha_i - \alpha_i'$, while 
$S_i \sim x^{\alpha_i'}$ on rectangles where $y \sim x^m$.  If $\beta_i' > 1$, by the last paragraph we have $2 \leq \beta_i' \leq d_i \leq {o \over 2}$ if
$m_p < 2$ and $2 \leq \beta_i' \leq d_i \leq o$ if $m_p \geq 2$. Furthermore, if $-{1 \over m}$ denotes the slope of this edge, then $m \geq m_p$, so if 
$m < 2$ we have $\beta_i' \leq {o \over 2}$ and if $m \geq 2$ we have  $\beta_i' \leq o$. Hence we are in case 2 of this theorem whenever $\beta_i' \geq 2$.

Thus we have proved Theorem 3.4 in all cases and we are done. \end{proof}

Suppose now we are in case 2 of Theorem 3.4. Since 
$\partial_{yy}S_i$ is monomialized by our resolution of singularities process,  we may let $\alpha_i''$, $\beta_i''$ be such that $\partial_{yy}S_i \sim
x^{\alpha_i''}y^{\beta_i''}$ in the final coordinates. The following theorem gives us estimates we will need later in the paper.

\begin{theorem} Suppose $i$ corresponds to case 2 of Theorem 3.4, and $\alpha_i$ is as in that theorem, so that $S_i \sim x^{\alpha_i}$ in the final
coordinates. There
 is a constant $C$ such that for each $j$ and $k$ for which $\psi_{ijk}$ is defined we have
\[\int_{[2^{-j}, 2^{-j+1}]\times [2^{-k}, 2^{-k+1}]} |x^{\alpha_i''}y^{\beta_i''  + 2 }|^{-{2 \over 3}\eta}|x^{\alpha_i}|^{-{\eta \over 3}}\,dx\,dy < C\tag{3.3}\]
\end{theorem}

\begin{proof} Let $\alpha_i'$, and $\beta_i'$ be as in Theorem 3.4. Thus $S_i(x,y) \sim x^{\alpha_i}$ on $[2^{-j}, 2^{-j+1}]\times [2^{-k}, 2^{-k+1}]$ and
the lowest edge of the Newton polygon $N(S_i)$ joins $(\alpha_i, 0)$ to $(\alpha_i',\beta_i')$, where $\beta_i' \geq 2$. Since $\beta_i' \geq 2$, 
$(\alpha_i',\beta_i')$ is also a vertex of $N(y^2 \partial_{yy} S_i)$. Since $\partial_{yy}S_i$ is monomialized by our resolution of singularities process, 
$y^2\partial_{yy}S_i$ is also monomialized, comparable to $x^{\alpha_i''}y^{\beta_i''  + 2 }$. Since $(\alpha_i',\beta_i')$ is a vertex of
 $N(y^2 \partial_{yy} S_i)$, $x^{\alpha_i'}y^{\beta_i'}$ is either
comparable to $|y^2\partial_{yy}S_i(x,y)| \sim x^{\alpha_i''}y^{\beta_i'' + 2}$ on $[2^{-j}, 2^{-j+1}]\times [2^{-k}, 2^{-k+1}]$, or is dominated by
 $x^{\alpha_i''}y^{\beta_i'' + 2}$ on $[2^{-j}, 2^{-j+1}]\times [2^{-k}, 2^{-k+1}]$. In either case, there is a constant $C'$ such that on
 $[2^{-j}, 2^{-j+1}]\times [2^{-k}, 2^{-k+1}]$ we have
\[x^{\alpha_i'}y^{\beta_i'} \leq C'x^{\alpha_i''}y^{\beta_i'' + 2} \tag{3.4}\]
Consequently, we have
\[\int_{[2^{-j}, 2^{-j+1}]\times [2^{-k}, 2^{-k+1}]} |x^{\alpha_i''}y^{\beta_i''  + 2 }|^{-{2 \over 3}\eta}|S_i|^{-{\eta \over 3}} < C'' 
\int_{[2^{-j}, 2^{-j+1}]\times [2^{-k}, 2^{-k+1}]}(x^{-{2 \over 3}\alpha_i'\eta}y^{-{2 \over 3}\beta_i'\eta}) (x^{-{\eta \over 3}\alpha_i}) \tag{3.5}\]
Thus we need to show that the integrand on the right in $(3.5)$ is uniformly integrable over $j$ and $k$. The main issue is to show that 
${2 \over 3}\beta_i'\eta \leq 1$, which will imply that the integral on the right of $(3.5)$ is nonincreasing in $k$. 

\begin{lemma} ${2 \over 3}\beta_i'\eta \leq 1$ for all $i$.
\end{lemma}
\begin{proof} Let $-{1 \over m}$ denote the slope of the edge $e$ of $N(S_i)$ joining $(\alpha_i, 0)$ to $(\alpha_i',\beta_i')$. Then when $y \sim x^m$,
the functions $x^{\alpha_i}$ and $x^{\alpha_i'}y^{\beta_i'}$ are of comparable magnitude. In particular, the integral of 
$(x^{\alpha_i'}y^{\beta_i'})^{-\eta}$ will
be uniformly bounded over dyadic rectangles where $y \sim x^m$, since the same is true for $|S_i|^{-\eta} \sim x^{-\alpha_i\eta}$. This translates into the statement that $-\eta(\alpha_i' + m\beta_i') +  m + 1
\geq 0$, or equivalently that $\eta \leq {m + 1 \over \alpha_i' + m\beta_i'}$. When $m \geq 2$ this implies that 
\[\eta \leq {m + 1 \over \alpha_i' + m\beta_i'} \leq {m + 1 \over m \beta'} \leq {3 \over 2\beta_i'}\]
This gives the lemma when $m \geq 2$. When $m < 2$, case 2 of Theorem 3.4 says that
$\beta_i' \leq {o \over 2}$. But any function in two variables satisfies $\eta \leq {2 \over o}$ since this is the exponent for $|x|^o + |y|^o$.
 As a result, if $m < 2$ we have ${2 \over 3}\beta_i'\eta  \leq {2 \over 3}$, better
than the estimate needed. This completes the proof of the lemma.
\end{proof}

We can now complete the proof of Theorem 3.5 in short order.  Since 
$|S| \sim x^{\alpha_i}$, the vertex $(0,\alpha_i)$ dominates $(\alpha_i', \beta_i')$ in the sense that for
some constant $C$ we have $x^{\alpha_i} \geq Cx^{\alpha_i'}y^{\beta_i'}$ on $D_i'$. This translates into the statement that $y < C' x^m$ throughout
 our domain
for some constant $C'$. Consequently, since by Lemma 3.6 the integral on the right of $(3.5)$ is nonincreasing in $k$, it can never be larger
than when $y \sim x^m$. But $x^{\alpha_i}$ and $x^{\alpha_i'}y^{\beta_i'}$ are of comparable magnitude when $y \sim x^m$. Hence in this
situation, the right hand integral of $(3.5)$ is simply comparable to $\int_{{[2^{-j}, 2^{-j+1}]\times [2^{-k}, 2^{-k+1}]}} |S_i|^{-\eta}\,dx\,dy$, which is uniformly bounded in $j$ and $k$
by the definition of $\eta$ as the supremum of the exponents making such integrals finite. This concludes the proof of Theorem 3.5.

\end{proof}

\section{The proof of Theorem 1.1.}

\subsection{The decomposition of the operator}

Let $\rho_0(x)$ be a function on $\R^3$ whose Fourier transform is nonnegative, compactly supported, and equal to one on a neighborhood of the origin. Let $\rho(x) =
8\rho_0(2x) - \rho_0(x)$. Then one can write $\delta(x) = \rho_0(x) + \sum_{n=0}^{\infty} 2^{3n} \rho(2^nx)$, and correspondingly we write
\[Tf = T_0f + \sum_{n=0}^{\infty} T_nf \tag{4.1}\]
Here $T_0 f = Tf \ast \rho_0$ and $T_nf = Tf \ast  2^{3n} \rho(2^nx)$. $T_0$ immediately satisfies the estimates of Theorem 1.1 and more, so we 
need only consider $T_n$ for $n > 0$.
Let $\eta_1 = \min(\eta, {2 \eta' \over 1 + 2\eta'})$ as in Theorem 1.1. Then the vertices of the trapezoid in Theorem 1.1 are
$(0,0), ({\eta_1 \over 2}, \eta_1), (1 - {\eta_1 \over 2}, \eta_1)$, and $(1,0)$. For each $\epsilon > 0$, we will exhibit a $p(\epsilon) > 0$ and
a $\delta(\epsilon) > 0$  such that 
such that $||T_n f||_{L^{p(\epsilon)}_{\eta_1 - \epsilon}} \leq C_{\epsilon} 2^{-\delta(\epsilon)n} ||f||_{L^{p(\epsilon)}}$. Adding over all 
$n$ gives $||T f||_{L^{p(\epsilon)}_{\eta_1 - \epsilon}} \leq C_{\epsilon}' ||f||_{L^{p(\epsilon)}}$.
Duality then gives the corresponding estimate for $p'(\epsilon)$ satisfying ${1 \over p(\epsilon)} + {1 \over p'(\epsilon)} = 1$.
As $\epsilon$ approaches zero, ${1 \over p(\epsilon)}$ will approach ${\eta_1 \over 2}$, and the above estimates will then imply Theorem 1.1.

\noindent Next, we write $T_n f = f \ast \mu_n$, where $\mu_n$ is the measure whose Fourier transform satisfies
\[\widehat{\mu_n}(\lambda) = \widehat{\rho}(2^{-n} \lambda) \int e^{-i\lambda_1 x - i\lambda_2y - i\lambda_3S(x,y)}\phi(x,y)\,dx\,dy \tag{4.2}\]
Let $\alpha_1(x)$ be a smooth nonnegative compactly supported function on $\R$ that is equal to 1 on a neighborhood of $0$, and let $\alpha_2(x) = 
1 - \alpha_1(x)$. For a $c > 0$ to be determined
by our arguments, we write $\mu_n = \mu_n^1 + \mu_n^2$, where for $k = 1, 2$ we have
\[\widehat{\mu_n^k} (\lambda) = \widehat{\rho}(2^{-n} \lambda)\alpha_k(c2^{-n}\lambda_3)\int e^{-i\lambda_1 x - i\lambda_2y - i\lambda_3S(x,y)}\phi(x,y)\,dx\,dy \tag{4.3}\]
We correspondingly let $T_n = T_n^1 + T_n^2$, where $T_n^k f = f \ast \mu_n^k$. If $\phi$ is supported on a sufficiently small neighborhood of the
origin, the operator $T_n^1$ is readily seen to be bounded from $L^2$ to any $L^2_s$ for $s > 0$ due to the cutoff function $\alpha_1(c2^{-n}\lambda_3)$
present in $(4.3)$. To see why, suppose for example $\lambda_1 > \lambda_2$. Then the $\alpha_1(c2^{-n}\lambda_3)$ factor ensures that  
$|\lambda_3|/|\lambda|$, and therefore $|\lambda_3|/|\lambda_1|$, is bounded above. Thus if $\phi$ is supported on a sufficiently small neighborhood 
of the origin, the phase function in $(4.3)$ has $x$ derivative bounded below by some $c'|\lambda_1| > c''|\lambda|$, and one can integrate by parts
in $x$ repeatedly in $(4.3)$ to get an estimate $|\widehat{\mu_n^1} (\lambda)| \leq C_s|\lambda|^{-s}$ holding for any $s$.

Thus $T_n^1$ is bounded from $L^2$ to any $L^2_s$. Since $|\lambda| \sim 2^{-n}$ in $(4.3)$, this immediately implies an
estimate  $||T_n^1 f||_{L^2_s} \leq C_s 2^{-n}||f||_{L^2}$  as well. Interpolating this with the $L^p$ to $L^p$ estimates for $p \in (1,\infty)$ for $s$ 
large enough will then give the estimates we need, and more.

So we focus our attention on $T_n^2$. To simplify notation, we write $\sigma(\lambda) = \widehat{\rho}(\lambda)\alpha_2(c\lambda_3)$, so that we have
\[\widehat{\mu_n^2} (\lambda) = \sigma(2^{-n} \lambda)\int e^{-i\lambda_1 x - i\lambda_2y - i\lambda_3S(x,y)}\phi(x,y)\,dx\,dy \tag{4.4}\]
Assume $\phi(x,y)$ is supported on a neighborhood of the origin on which Theorem 3.2 applies, and write $\phi(x,y) = \sum_{ijk}\phi_{ijk}(x,y)$ as in that
theorem. We correspondingly write $T_n^2 = \sum_{ijkn} U_{ijkn}$, where $U_{ijkn} f = f\ast \nu_{ijkn}$. Here $\nu_{ijkn}$ is the measure whose
Fourier transform satisfies
\[\widehat{\nu_{ijkn}} (\lambda) = \sigma(2^{-n} \lambda)\int e^{-i\lambda_1 x - i\lambda_2y - i\lambda_3S(x,y)}\phi_{ijk}(x,y)\,dx\,dy \tag{4.5}\]
Let $\eta_i$ be as in Theorem 3.1 in the situation at hand.  Let $\gamma_{ijk}(x,y) = \phi_{ijk} \circ \eta_i(x,y)$. Then
due to the form of the coordinate changes $(x,y) \rightarrow (x, y + h_i(x))$ of this paper, for some $h_i(x)$ equation $(4.5)$ becomes
\[\widehat{\nu_{ijkn}} (\lambda) = \sigma(2^{-n} \lambda)\int e^{-i\lambda_1 x  -  i\lambda_2y  - i\lambda_2 h_i(x) - i\lambda_3S_i(x,y)}\gamma_{ijk}(x,y)\,dx\,dy \tag{4.6}\]
Here $h_i(x)$ is a real analytic function of $x^{1 \over M}$ for some $M$ and $S_i(x,y)$ is the composition of $S(\pm x, \pm y)$ or $S(\pm y, \pm x)$ with
the map   $(x,y) \rightarrow (x, y + h_i(x))$.
Recall by Theorem 3.2 the $\gamma_{ijk}(x,y)$ are uniformly bounded and for some $C$ the function $\gamma_{ijk}(x,y)$ is supported on 
$[C^{-1}2^{-j}, C 2^{-j}] \times [C^{-1}2^{-k}, C 2^{-k}]$.
Hence for any $s$, the measure $(1 - \Delta)^{s \over 2} \nu_{ijkn}$ is a function with $L^1$ norm bounded by $C_s 2^{ns - j - k}$. Thus 
$(1 - \Delta)^{s \over 2}U_{ijkn} f$ is the convolution of $f$ with a function of $L^1$ norm bounded by $C_s 2^{ns - j - k}$. So by Young's inequality, 
$(1 - \Delta)^{s \over 2}U_{ijkn}$ is bounded on any $L^p$ with norm bounded by $C_s 2^{ns - j - k}$ as well. In other words $U_{ijkn}$ is bounded from 
$L^p$ to $L^p_s$ with norm bounded by $C_s 2^{ns - j - k}$.

In the proof at hand, we are considering only $s \leq \eta_1$. For any such $s$, by the above we have 
$||\sum_{j +  k \geq 2n\eta_1} U_{ijkn}||_{L^p \rightarrow L^p_s} \leq C_s' 2^{-\eta_1 n}$, which decreases exponentially in $n$ and therefore 
satisfies the estimates we seek. Hence it suffices to consider $||\sum_{j +  k < 2n\eta_1} U_{ijkn}||_{L^p \rightarrow L^p_s}$. There are only $O(n^2)$
such terms, so if we show that each  $||U_{ijkn}||_{L^p \rightarrow L^p_s} \leq C_s 2^{-\delta_s n}$  for the $(p,s)$ at hand, where $\delta_s > 0$, 
that will be enough to prove Theorem 1.1. This is the estimate we will prove.

\subsection{Defining the  $V_{ijkn}^l$ and $W_{ijkn}^l$}

Restating the above, the goal is to show for $\epsilon > 0$ that  $||U_{ijkn} f||_{L^{p(\epsilon)}_{\eta_1 - \epsilon}} \leq C_{\epsilon} 2^{-\delta(\epsilon)n} ||f||_{L^{p(\epsilon)}}$ where ${1 \over p(\epsilon)}$ approaches ${\eta_1 \over 2}$ as $\epsilon \rightarrow 0$.
We will make extensive use of the expression $(4.6)$ for $\widehat{\nu_{ijkn}} (\lambda)$. To simplify notation, we let $R(x,y) = S_i(x,y)$, so that $(4.6)$
becomes
\[\widehat{\nu_{ijkn}} (\lambda) = \sigma(2^{-n} \lambda)\int e^{-i\lambda_1 x  -  i\lambda_2y  - i\lambda_2 h_i(x) - i\lambda_3R(x,y)}\gamma_{ijk}(x,y)\,dx\,dy \tag{4.7}\]
Let $P(x,y)$ denote the phase function $\lambda_1 x + \lambda_2y + \lambda_2 h_i(x)  + \lambda_3R(x,y)$ of $(4.6)$. Note that $P_{yy}(x,y) = \lambda_3 R_{yy}(x,y)$, and also that
$R_{yy}(x,y) = \partial_y^2 S_i(x,y)$ is monomialized in the final coordinates. In other words, on the support of $\gamma_{ijk}(x,y)$, for some $\alpha_1$ and 
$\beta_1$ one has that $R_{yy} \sim x^{\alpha_1}y^{\beta_1} \sim 2^{-j\alpha_1 - k\beta_1}$. Thus for fixed $x$, one can view the $y$
 integral in $(4.6)$ as an integral of a function whose second derivative is roughly constant, and then use standard stationary phase on this integral. We
will take this tack. If $P$ has a critical point in such a $y$ integral, the natural width of a cycle of the phase at the critical point is 
$(|\lambda|2^{-j\alpha_1 - k\beta_1})^{-{1 \over 2}}$. (We can use $|\lambda|$ and not $|\lambda_3|$ here since $|\lambda| \sim |\lambda_3|$ on the
support of $\sigma(2^{-n}\lambda)$). Correspondingly, we will have two different arguments, essentially depending on whether or not  
$2^{-k} > (|\lambda|2^{-j\alpha_1 - k\beta_1})^{-{1 \over 2}}$, since the length of the $y$ interval of integration is $\sim 2^{-k}$. 

We will put the above philosophy into effect as follows. It will make our arguments somewhat technically easier if for a small but fixed $\epsilon_0 > 0$ 
we make the two cases 
$2^{-k} > |\lambda|^{-{1 \over 2} + {\epsilon_0 \over 8}}(2^{-j\alpha_1 - k\beta_1})^{-{1 \over 2}}$ and 
$2^{-k} \leq |\lambda|^{-{1 \over 2} + {\epsilon_0 \over 8}}(2^{-j\alpha_1 - k\beta_1})^{-{1 \over 2}}$. We refer to $U_{ijkn}$ corresponding to
the first situation as case 1 operators, and $U_{ijkn}$ corresponding to the second situation as case 2 operators. The argument for case 1 operators 
will have some resemblance to an argument that can be used for phases with nondegenerate Hessian determinant, and these operators will account for
why the index $\eta'$ appears in Theorem 1.1. The argument for case 2 operators doesn't use the Hessian determinant, and instead resembles
to some extent an argument that can be used to find uniform decay estimates for Fourier transforms of surface measures, and accounts for the index
$\eta$ appearing in Theorem 1.1.

We split case 1 operators as a sum $U_{ijkn} =  V_{ijkn}^1 + V_{ijkn}^2$ and case 2 operators as a sum $U_{ijkn} =  W_{ijkn}^1 + W_{ijkn}^2$ as 
follows. As we did earlier, we let $\alpha_1$ be a smooth nonnegative compactly supported function on $\R$ equal to $1$ on a neighborhood of $0$), and 
let $\alpha_2(x) = 1- \alpha_1(x)$. Define $R_{yy}^* = 2^{-j\alpha_1 - k\beta_1}$, which we view as a fixed number that is comparable to $R_{yy}$ on
the domain of integration in $(4.7)$. We first define $W_{ijkn}^l$ for $l = 1, 2$ by $W_{ijkn}^l f = f \ast \mu_{ijkn}^l$, where 
\[\widehat{\mu_{ijkn}^l} (\lambda) = \sigma(2^{-n} \lambda)\int e^{-i\lambda_1 x  -  i\lambda_2y  - i\lambda_2 h_i(x) - i\lambda_3R(x,y)}\gamma_{ijk}(x,y)\alpha_l\big(|\lambda|^{- {\epsilon_0 \over 4}}2^{-k} P_y(x,y)\big) \,dx\,dy \tag{4.8}\]
We next define $ V_{ijkn}^1$ and $ V_{ijkn}^2$. Our constructions can be done such that each $\gamma_{ijk}$ is supported on the 
union of boundedly many rectangle of dimensions $c2^{-j}$ by $c2^{-k}$ such that our various estimates hold on the rectangle of dimensions $c'2^{-j}$ 
by $c'2^{-k}$ for some $c' > c$. Since $P_{yy} \neq 0$ throughout, if for a given $x$ there is some $y$
for which $P_y(x,y) = 0$ in one of these enlarged rectangles, we may define $y^*(x)$ by the condition that $P_y(x,y^*(x)) = 0$.
We define $V_{ijkn}^1 f = f \ast \nu_{ijkn}^1$ and $V_{ijkn}^2 f = f \ast \nu_{ijkn}^2$, where for $l = 1,2$
we have
\[\widehat{\nu_{ijkn}^l} (\lambda) = \sigma(2^{-n} \lambda)\int e^{-i\lambda_1 x  -  i\lambda_2y  - i\lambda_2 h_i(x) - i\lambda_3R(x,y)}\gamma_{ijk}(x,y)\alpha_l\big(|\lambda|^{{1 \over 2} - {\epsilon_0 \over 16}}(R_{yy}^*)^{1 \over 2} (y - y^*(x)\big) \,dx\,dy \tag{4.9}\]
If $y^*(x)$ does not exist in any of the enlarged rectangles, we replace the $\alpha_l\big(|\lambda|^{{1 \over 2} - {\epsilon_0 \over 16}}
(R_{yy}^*)^{1 \over 2} (y - y^*(x)\big)$ factor by $0$ if $l =1$ and by $1$ if $l = 2$. 

Suppose $(x,y)$ is such that the factor given by $\alpha_1\big(|\lambda|^{{1 \over 2} - {\epsilon_0 \over 16}}
(R_{yy}^*)^{1 \over 2} (y- y^*(x)\big)$ is well-defined and nonzero. Then $|y - y^*(x)| \leq C |\lambda|^{-{1 \over 2} + {\epsilon_0 \over 16}}
(R_{yy}^*)^{-{1 \over 2}}$. Since we are in case 1, we have $2^{-k} > |\lambda|^{-{1 \over 2} + {\epsilon_0 \over 8}}(R_{yy}^*)^{-{1 \over 2}}$, so
that $|y - y^*(x)| \leq C|\lambda|^{-{\epsilon_0 \over 16}} 2^{-k}$.  We may assume that  $n$ is large enough so that the $C|\lambda|^{-{\epsilon_0 \over 16}} 2^{-k}$ factor here is less than $(c' - c)2^{-k}$. Thus
if $(x,y)$ is on the upper or lower boundary of one of the  $c2^{-j}$ by $c2^{-k}$ rectangles and  
$\alpha_1\big(|\lambda|^{{1 \over 2} - {\epsilon_0 \over 16}}
(R_{yy}^*)^{1 \over 2} (y- y^*(x)\big)$ is well-defined and nonzero, the monotone function $y^*(x)$ will still be defined  (in one of the $c'2^{-j}$ by
 $c'2^{-k}$ rectangles) as this factor ``makes its exit'' through the boundary. Thus in the $l = 1$ case,
the set of points on which $y^*(x)$ is defined contains a union of boundedly many intervals containing the points $x$ for which the integrand in $(4.9)$ is nonzero for any $y$.

If $x$ is such that $y^*(x)$ does not exist in the $c'2^{-j}$ by $c'2^{-k}$ rectangles, then for any $(x,y)$ in the support of $\gamma_{ijk}$, for such an $x$
we must have that $P_y(x,y') \neq 0$ for any $y'$ with $|y' - y| < (c' - c)2^{-k}$. Since $|P_{yy}| > C'|\lambda| R_{yy}^*$, this means that $|P_y(x,y)| > C''|\lambda|R_{yy}^* 2^{-k} > C''' |\lambda|^{{1 \over 2}  + {\epsilon_0 \over 8}}(R_{yy}^*)^{1 \over 2}$, the latter inequality following from the fact 
that we are in case 1.

\subsection{The analysis of the operators $V_{ijkn}^2$ and $W_{ijkn}^2$}

The idea behind the analysis of $V_{ijkn}^2$ and $W_{ijkn}^2$ are similar. We simply repeatedly integrate by parts in the $y$ variable in $(4.9)$.
The $\alpha_2$ factors in both integrals are such that each time we do such an integration by parts, one gains a factor of $C|\lambda|^{-{\epsilon_0 \over 16}}$. Thus doing it enough times will give that both $|\widehat{\mu_{ijkn}^2} (\lambda)|$ and $|\widehat{\nu_{ijkn}^2} (\lambda)|$ decay faster
than any negative power of $|\lambda|$. Hence $V_{ijkn}^2$ and $W_{ijkn}^2$ are bounded from $L^2$ to any $L^2_s$ with operator norm bounded by
say $C2^{-n}$. Thus interpolating with the $L^p$ to $L^p$ bounds for $1 < p < \infty$ gives the desired $L^p$ to $L^p_s$ bounds, and more.

We start with  $V_{ijkn}^2$. If $y^*(x)$ exists, then in the support of the integrand of $(4.9)$, one has $|y - y^*(x)| \geq C|\lambda|^{-{1 \over 2} + {\epsilon_0 \over 16}}
(R_{yy}^*)^{-{1 \over 2}}$. Since $P_y(x,y^*(x)) = 0$, for such a $y$ we have $|P_y(x,y)| = |P_y(x,y) - P_y(x,y^*(x))| 
\geq C'|\lambda_3|R_{yy}^* |y - y^*(x)|$. Since $|\lambda_3| \sim |\lambda|$ on the domain in question, the lower bounds on $|y - y^*(x)|$ lead to
\[ |P_y(x,y)| \geq C |\lambda|^{{1 \over 2} + {\epsilon_0 \over 16}} (R_{yy}^*)^{{1 \over 2}} \tag{4.10}\]
Equation $(4.10)$ will hold whenever the integrand in $(4.9)$ is nonzero for $l = 2$ and $y^*(x)$ exists. But even if $x$ is such that $y^*(x)$ does not exist, 
$(4.10)$ will still hold for each $y$ for which $\gamma_{ijk}(x,y) \neq 0$ by the discussion at the end of the last section.
 We now integrate by parts repeatedly in $(4.9)$ in $y$, integrating 
$-iP_y(x,y)e^{-iP(x,y)}$ in $y$ then differentiating $-{1 \over i P_y(x,y)}$ times the remaining factors. Each time we do so, we get a factor of
$|{1 \over i P_y(x,y)}|$ that is bounded by  $C'''|\lambda|^{-{1 \over 2} - {\epsilon_0 \over 16}} (R_{yy}^*)^{-{1 \over 2}}$, and various other factors 
that depend on where the derivative lands. We examine each possibility in this regard.

Each time the derivative lands on a negative power of $P_y(x,y)$, the integration by parts results in a factor bounded by
 $C\big|{P_{yy}(x,y) \over (P_y(x,y))^2 }\big|$. Note that
$|P_{yy}(x,y)| = |\lambda_3 R_{yy}(x,y)| \leq C|\lambda| R_{yy}^*$, so in view of $(4.10)$ this factor is bounded by 
$C|\lambda|^{-{\epsilon_0 \over 8}}$. If the derivative lands on some derivative of the $\alpha_2$ factor in $(4.9)$, the integration by parts results
in a factor bounded by $C |\lambda|^{{1 \over 2} - {\epsilon_0 \over 16}}(R_{yy}^*)^{1 \over 2} | (P_y(x,y))|^{-1}$. The support of this derivative
of the $\alpha_2$ factor is such that $(4.10)$ holds, so we get a bound of $C'|\lambda|^{-{\epsilon_0 \over 8}}$ again. 

If the derivative lands on
some $y$ derivative of $\gamma_{ijk}(x,y)$, then by part 2 of Theorem 3.2, we incur a factor bounded by $C2^k| P_y(x,y)|^{-1}$, which by
$(4.10)$ is at most $C'2^k |\lambda|^{-{1 \over 2} -  {\epsilon_0 \over 16}} (R_{yy}^*)^{-{1 \over 2}}$. Since we are in case $1$, we have
$2^k <  |\lambda|^{{1 \over 2} - {\epsilon_0 \over 8}}(R_{yy}^*)^{1 \over 2}$, and we have a bound of $C''|\lambda|^{-{3  \over 16}\epsilon_0}$,
better than what we need. Lastly, if the derivative lands on some $\partial_y^m R(x,y)$ for some $m \geq 2$, we go from an $m$th derivative to an
$m+1$th derivative of $R(x,y)$. By Corollary 3.1.2, each such derivative results in a factor of $C|y|^{-1} < C' 2^k$ on top of the bound $C''R_{yy}^*$ for 
the second derivative. Thus once again we incur a factor bounded by a constant times $C2^{-k}| P_y(x,y)|^{-1}$, which we saw above is bounded by
$C''|\lambda|^{-{3  \over 16}\epsilon_0 }$.

In summary, after each integration by parts, we gain a factor of at least $C|\lambda|^{-{\epsilon_0 \over 8}}$. Thus as explained in the beginning of this
section, by integrating by parts enough times and then interpolating we can get any $L^p$ to $L^p_s$ boundedness we desire. This concludes the 
analysis of the operators $V_{ijkn}^2$.

We now do a similar analysis for the $W_{ijkn}^2$. This time, in view of the support of the $\alpha_2$ factor in $(4.8)$, we replace $(4.10)$  by
\[|P_y(x,y)| \geq C|\lambda|^{\epsilon_0 \over 4} 2^k \tag{4.11}\]
We now perform the same repeated integration by parts as before, and examine the effect of the derivative landing on each possible factor. If it lands on a 
negative power of $P_y(x,y)$, we incur a factor bounded by $C\big|{P_{yy}(x,y) \over (P_y(x,y))^2 }\big|$, which in view of $(4.11)$ and the fact that
$P_{yy} = \lambda_3 R_{yy}$ is bounded by
$C' R_{yy}^*|\lambda|^{1 - {\epsilon_0 \over 2}} 2^{-{2k}}$. Since we are in case 2, we have $2^{-k} \leq |\lambda|^{-{1 \over 2} + {\epsilon_0 \over 8}}(R_{yy}^*)^{-{1 \over 2}}$. Hence the factor is bounded by $C'' |\lambda|^{-{\epsilon_0 \over 4}}$, which gives us what we need.

Next, if the derivative lands on a derivative of the $\alpha_2$ factor in $(4.8)$, then the integration by parts incurs a factor bounded by
$C|\lambda|^{1 -{\epsilon_0 \over 4}}2^{-k} R_{yy}^* | P_y(x,y)|^{-1}$. By $(4.11)$ this is at most $C'|\lambda|^{1 -{\epsilon_0 \over 2}}
2^{-2k} R_{yy}^*$, the same bound as before. So once again we have a bound of $C''|\lambda|^{-{\epsilon_0 \over 4}}$.

If the derivative lands on some $y$ derivative of $\gamma_{ijk}(x,y)$, then by part 2 of Theorem 3.2, we incur a factor bounded by
 $C2^k| P_y(x,y)|^{-1}$, which in view of $(4.11)$ is at most $C'|\lambda|^{-{\epsilon_0 \over 4}}$, as needed. Lastly, if the derivative 
  lands on a $\partial_y^m R(x,y)$ for some $m \geq 2$, so that we go from an $m$th derivative to an
$m+1$th derivative of $R(x,y)$, exactly as in the case of the $V_{ijkn}^2$ operators we again incur a factor of $C2^k| P_y(x,y)|^{-1}$, thus again giving
us the bound $C|\lambda|^{-{\epsilon_0 \over 4}}$ we need.

In summary, after each integration by parts, we gain a factor of at least $C|\lambda|^{-{\epsilon_0 \over 4}}$. So just as in the case of the operators
$V_{ijkn}^2$, by integrating by parts enough times we can get any negative power of $|\lambda|$ we seek, and therefore using interpolations we 
can get any $L^p$ to $L^p_s$ boundedness statement. This concludes the analysis of the operators $W_{ijkn}^2$.

\subsection{$L^p$ to $L^p_s$ bounds for the operators $V_{ijkn}^1$} 

We will be making use of the Van der Corput lemma for oscillatory integrals (see p.334 of [S1] for a proof.)

\begin{lemma}
Suppose $h(x)$ is a real-valued $C^k$ function on the interval $[a,b]$ such that $|h^{(k)}(x)| > A$ on $[a,b]$ for
some $A > 0$. Let $\phi(x)$ be $C^1$ on $[a,b]$. 

\noindent If $k \geq 2$ there is a constant $c_k$ depending only on $k$ such that
\[\bigg|\int_a^b e^{ih(x)}\phi(x)\,dx\bigg| \leq c_kA^{-{1 \over k}}\bigg(|\phi(b)| + \int_a^b |\phi'(x)|\,dx\bigg) \]
If $k =1$, the same is true if we also assume that $h(x)$ is $C^2$ and $h'(x)$ is monotone on $[a,b]$. 
\end{lemma}

\noindent We examine $(4.9)$ in the $l = 1$ case, which we may write as
\[\widehat{\nu_{ijkn}^1} (\lambda) = \sigma(2^{-n} \lambda)\int e^{-iP(x,y)}\gamma_{ijk}(x,y)\alpha_1\big(|\lambda|^{{1 \over 2} - {\epsilon_0 \over 16}}(R_{yy}^*)^{1 \over 2} (y - y^*(x)\big) \,dx\,dy \tag{4.12}\]
Next, shift coordinates in $y$ in $(4.12)$, replacing $y$ by $y + y^*(x)$. The result is
\[\widehat{\nu_{ijkn}^1} (\lambda) = \sigma(2^{-n} \lambda)\int e^{-iP(x,y + y^*(x))}\gamma_{ijk}(x,y + y^*(x))\alpha_1\big(|\lambda|^{{1 \over 2} - {\epsilon_0 \over 16}}(R_{yy}^*)^{1 \over 2}\, y \big) \,dx\,dy \tag{4.13}\]
The right-hand side of this may be written as
\[ \sigma(2^{-n} \lambda)\int e^{-iP(x,y^*(x))} e^{-i(P(x,y + y^*(x)) - P(x,y^*(x))}\gamma_{ijk}(x,y + y^*(x))\alpha_1\big(|\lambda|^{{1 \over 2} - {\epsilon_0 \over 16}}(R_{yy}^*)^{1 \over 2}\, y \big) \,dx\,dy \tag{4.14}\]
The idea now is that due to the $\alpha_1\big(|\lambda|^{{1 \over 2} - {\epsilon_0 \over 16}}(R_{yy}^*)^{1 \over 2}\, y \big)$ factor in $(4.14)$, which
has the effect of localizing the phase to near $y = 0$, the phase in $(4.14)$ is effectively $P(x,y^*(x))$ and we may bound $(4.14)$ well by simply using
stationary phase in the $x$ variable. We will do this via an application of the Van der Corput lemma. (Recall that the $x$ for which
the integrand of $(4.12)$ is nonzero for any $y$ is contained in finitely many intervals on which $y^*(x)$ is defined, so that issues regarding with the 
domain of $y^*(x)$ will not interfere with applying the Van der Corput lemma here in the $x$ variable.)
 Our first step is to observe that we have
\[\partial_x (P(x,y^*(x))) = P_x(x,y^*(x)) + (y^*)'(x) P_y(x,y^*(x))\]
\[ = P_x(x,y^*(x)) \tag{4.15}\]
The latter equation follows from the fact that $y^*(x)$ is defined through the the condition that $P_y(x,y^*(x)) = 0$. Taking a second derivative, we have
\[\partial_{xx}(P(x,y^*(x)))= P_{xx}(x,y^*(x)) + (y^*)'(x) P_{xy}(x,y^*(x)) \]
Since we have the formula $(y^*)'(x) = -{P_{xy}(x,y^*(x)) \over P_{yy}(x,y^*(x))}$, the above becomes
\[\partial_{xx}(P(x,y^*(x)))= {P_{xx}(x,y^*(x))P_{yy}(x,y^*(x)) - ( P_{xy}(x,y^*(x)))^2 \over P_{yy}(x,y^*(x))} \tag{4.16}\]
Note the presence of the Hessian determinant of $P(x,y^*(x))$ in the numerator of $(4.16)$. The Hessian determinant of $P(x,y)$ is invariant under transformations
of the form $(x,y) \rightarrow (x, y + h(x))$ at any point where $\partial_y P(x,y) = 0$, as can be verified by direct calculation; this holds even when not at a critical point.
 So the numerator in $(4.16)$ is the same as the Hessian determinant
of $P(x,y)$ in the original coordinates, before the resolution of singularities procedure was done. But in the original coordinates, $P(x,y)$ was 
$\lambda_1 x + \lambda_2 y + \lambda_3 S(x,y)$, whose Hessian is just $(\lambda_3)^2 H(x,y)$, where $H(x,y)$ denotes the Hessian of $S(x,y)$. 
Recall that our resolution of singularities algorithm
was applied to this Hessian, along with several other functions. 
So if we denote the Hessian determinant in the final coordinates by $H_i(x,y)$, the numerator in $(4.16)$ is just 
$(\lambda_3)^2 H_i(x,y)$, where $|H_i(x,y)|$ is comparable to a fixed value throughout the domain of $\gamma_{ijk}(x)$, which we denote by $H_i^*$.
Thus the absolute value of the right-hand side of $(4.16)$ is comparable to ${|\lambda_3|^2 H_i^* \over |\lambda_3| R_{yy}^*} \sim
 |\lambda| H_i^* ( R_{yy}^*)^{-1}$, using that $|\lambda_3| \sim |\lambda|$ here.

We now have uniform lower bounds on the absolute value of the second $x$ derivative of the phase function in $(4.14)$. As a result, we may 
apply the Van der Corput lemma to the $x$ integral in $(4.14)$. The result is this integral is bounded by $|\lambda|^{-{1 \over 2}}
(H_i^*)^{-{1 \over 2}} ( R_{yy}^*)^{{1 \over 2}}$ times the integral of the absolute value of the $x$ derivative of the product of the right-hand factors in $(4.14)$, which we will see is uniformly bounded in $j$ and $k$. 

The $x$ derivative can land in two places and get a nonzero value, the factor
 $\gamma_{ijk}(x,y + y^*(x))$ and the factor $e^{-i(P(x,y + y^*(x)) - P(x,y^*(x))}$.  We first deal with $\gamma_{ijk}(x,y + y^*(x))$. Note that we have
\[\partial_x(\gamma_{ijk}(x,y + y^*(x))) = (\gamma_{ijk})_x(x, y + y^*(x)) - (\gamma_{ijk})_y(x, y + y^*(x)){P_{xy}(x,y^*(x)) \over P_{yy}(x,y^*(x))}
\]
Since $P_{xy} = \lambda_3 R_{xy}$ and $P_{yy} = \lambda_3 R_{yy}$, the above can be rewritten as
\[\partial_x(\gamma_{ijk}(x,y + y^*(x))) = (\gamma_{ijk})_x(x, y + y^*(x)) - (\gamma_{ijk})_y(x, y + y^*(x)){R_{xy}(x,y^*(x)) \over R_{yy}(x,y^*(x))}
\tag{4.17}\]
By part 2 of Theorem 3.2, the first term in $(4.17)$ is bounded in absolute value by $C2^j$. Since the domain of integration has length at most
$C'2^{-j}$, its integral in $x$ is uniformly bounded as needed. As for the second term in $(4.17)$, using part 2 of Theorem 3.2 again we see it is 
bounded in absolute value by $C2^k {R_{xy}(x,y^*(x)) \over R_{yy}(x,y^*(x))}$. Since we resolved the singularities of $\partial_y(S(x,y) - S(0,y))$ 
(after a linear coordinate change) in
the resolution of singularities process, $R_{xy}(x,y^*(x))$ is monomialized in the final coordinates, so in particular there is some fixed 
number $R_{xy}^*$ such that $R_{xy} \sim R_{xy}^*$ whenever the integrand in $(4.17)$ is nonzero. Hence the ratio ${R_{xy}(x,y^*(x)) \over R_{yy}(x,y^*(x))}$ is comparable in magnitude to ${R_{xy}^* \over R_{yy}^*}$, and the second term is bounded by $C2^k{R_{xy}^* \over R_{yy}^*}$.

The length of the $x$ interval of integration in $(4.14)$ is at most the $x$-diameter of the set in which the integrand of $(4.12)$ is nonzero.
The $\alpha_1\big(|\lambda|^{{1 \over 2} - {\epsilon_0 \over 16}}(R_{yy}^*)^{1 \over 2} (y - y^*(x)\big)$ factor in $(4.12)$ implies bounds for
this diameter. The curve $y = y^*(x)$ has slope $\sim {R_{xy}^* \over R_{yy}^*}$, so it traverses an interval of $y$-length $C2^{-k}$ over an
$x$-distance bounded
by $C2^{-k}{R_{yy}^* \over R_{xy}^*}$. We saw after $(4.9)$ that the $y$-length of the support of $\alpha_1\big(|\lambda|^{{1 \over 2} - {\epsilon_0 \over 16}}(R_{yy}^*)^{1 \over 2} (y - y^*(x)\big)$ is less than $C|\lambda|^{-{\epsilon_0 \over 16}}2^{-k}$, so the factor $\alpha_1\big(|\lambda|^{{1 \over 2} - {\epsilon_0 \over 16}}(R_{yy}^*)^{1 \over 2} (y - y^*(x)\big)$ will also be supported on a set of $x$ diameter bounded by $C2^{-k}{R_{yy}^* \over R_{xy}^*}$.
 Thus the integral of the second term of $(4.17)$ is bounded by
$C 2^k{R_{xy}^* \over R_{yy}^*} \times 2^{-k} {R_{yy}^* \over R_{xy}^*}$ and is therefore uniformly bounded. 
We conclude that when the $x$ derivative lands on the $\gamma_{ijk}(x,y + y^*(x))$, the resulting term integrates to something uniformly bounded in
$j$ and $k$. 
 
We move now to the case where the derivative lands on  $e^{-i(P(x,y + y^*(x)) - P(x,y^*(x))}$. The resulting term has magnitude bounded by
\[C\big|[P_x(x, y + y^*(x)) - P_x(x, y^*(x))] - [P_y(x,y + y^*(x)) - P_y(x,y^*(x))] {R_{xy}(x,y^*(x)) \over R_{yy}(x,y^*(x))}\big| \tag{4.18}\]
Note that 
\[P_x(x, y + y^*(x)) - P_x(x, y^*(x)) = P_{xy}(x,y^*(x))y + O(y^2\sup P_{xyy})\]
\[ = \lambda_3  (R_{xy}(x,y^*(x))y + O(|\lambda| y^2\sup|R_{xyy}|) \tag{4.19}\]
Here the supremum is over the support of the integrand in $(4.14)$. Similarly, one has
\[P_y(x,y + y^*(x)) - P_y(x,y^*(x)) =   \lambda_3  (R_{yy}(x,y^*(x))y + O(|\lambda| y^2\sup |R_{yyy}|) \tag{4.20}\]
Inserting $(4.19)$ and $(4.20)$ into $(4.18)$, we see that $(4.18)$ is bounded by
\[O(|\lambda| y^2\sup|R_{xyy}|) +  O(|\lambda| y^2{R_{xy}^* \over R_{yy}^*}\sup |R_{yyy}|) \tag{4.21}\]
By Corollary 3.1.1, we have that $\sup |R_{yyy}| \leq C2^k R_{yy}^*$ and $\sup|R_{xyy}| \leq C 2^k R_{xy}^*$. (Technically, since we have resolved
the singularities of $\partial_y S(x,y) - \partial_y S(0,y)$ and not $\partial_{xy}S(x,y)$, Corollary 3.1.1 applies to $\partial_y R(x,y) - \partial_y R(x,0)$,
but it is not hard to show it applies to $\partial_{xy}R(x,y)$ as well since it is monomialzed in the final coordinates.) 
 Thus $(4.21)$ is bounded by
\[C|\lambda|y^2 2^k R_{xy}^* \tag{4.22}\]
As we saw above, the $x$ interval of integration in $(4.14)$ has length at most $C2^{-k}{R_{yy}^* \over R_{xy}^*}$. Hence by $(4.22)$, the
$x$ integral of the term in question is at most $C'|\lambda|y^2 2^k R_{xy}^*  \times 2^{-k}{R_{yy}^* \over R_{xy}^*} = C'|\lambda|y^2R_{yy}^*$.
Due to the presence of the $\alpha_1\big(|\lambda|^{{1 \over 2} - {\epsilon_0 \over 16}}(R_{yy}^*)^{1 \over 2}\, y \big)$ factor in $(4.14)$, we have
\[|y| \leq C|\lambda|^{-{1 \over 2} + {\epsilon_0 \over 16}}(R_{yy}^*)^{-{1 \over 2}}\tag{4.23}\]
Thus $C'|\lambda|y^2R_{yy}^* \leq C'' |\lambda|^{\epsilon_0 \over 8}$, another bound that is uniform in $j$ and $k$.

We conclude that we can indeed apply the Van der Corput lemma to the $x$ integral in $(4.14)$, obtaining a bound of $C |\lambda|^{\epsilon_0 \over 8}$
times $|\lambda|^{-{1 \over 2}}(H_i^*)^{-{1 \over 2}} ( R_{yy}^*)^{{1 \over 2}}$. We then can do the $y$ integration, using $(4.23)$. We get that
expression $(4.14)$ for $\widehat{\nu_{ijkn}^1} (\lambda)$ is at most $C|\lambda|^{{3 \over 16}\epsilon_0} |\lambda|^{-1}(H_i^*)^{-{1 \over 2}}$.
Since $\epsilon_0 > 0$ was arbitrary, $|\lambda| \sim 2^n$, and and $V_{ijkn}^1 f = f \ast \nu_{ijkn}^1$, this can be restated as the statement that for every $\epsilon > 0$ there is 
a $\delta(\epsilon) > 0$ such that $||V_{ijkn}^1||_{L^2 \rightarrow L^2_{1 - \epsilon}} \leq C_{\epsilon}2^{-\delta(\epsilon)n} (H_i^*)^{-{1 \over 2}}$.
Since $H_i(x,y)$ is comparable to $H_i^*$ on $[2^{-j}, 2^{-j+1}] \times [2^{-k}, 2^{-k+1}]$ and $\gamma_{ijk}$ is supported in a bounded
dilation of this box, we may restate this as 
\[||V_{ijkn}^1 f||_{L^2 \rightarrow L^2_{1 - \epsilon}}  \leq C_{\epsilon}2^{-\delta(\epsilon)n} \int_{[2^{-j}, 2^{-j+1}] \times [2^{-k}, 2^{-k+1}]}
|H_i(x,y) x^2 y^2 |^{-{1 \over 2}}\tag{4.24}\]
On the other hand, due to the location of the support of $\gamma_{ijk}$, the measure $\nu_{ijkn}^1$ has $L^1$ norm bounded by $C2^{-j - k} 
= C\int_{[2^{-j}, 2^{-j+1}] \times [2^{-k}, 2^{-k+1}]}|H_i(x,y) x^2 y^2 |^0$. Thus by Young's inequality we have
\[||V_{ijkn}^1 f||_{L^{(\epsilon^{-1})} \rightarrow L^{({\epsilon}^{-1})}} \leq C \int_{[2^{-j}, 2^{-j+1}] \times [2^{-k}, 2^{-k+1}]}
|H_i(x,y) x^2 y^2 |^0  \tag{4.25}\]
We now interpolate $(4.24)$ and $(4.25)$, using weighting $L^2$ by $\eta_1 = \min(\eta, {2 \eta' \over 1 + 2\eta'})$ as in the statement of Theorem 1.1,
and weighting $L^{{\epsilon}^{-1}}$ by $1 - \eta_1$. In view of the fact that $|H_i(x,y) x^2 y^2 |$ is comparable to a fixed value on the domain of
integration in $(4.24)$ or $(4.25)$, the result is
\[||V_{ijkn}^1 f||_{L^{{2 \over \eta_1} + \delta'(\epsilon)} \rightarrow L^{{2 \over \eta_1} + \delta'(\epsilon)}_{\eta_1 - \eta_1\epsilon}}  \leq C_{\epsilon}'2^{-\delta'(\epsilon)n} \int_{[2^{-j}, 2^{-j+1}] \times [2^{-k}, 2^{-k+1}]}
|H_i(x,y) x^2 y^2 |^{-{\eta_1 \over 2}} \tag{4.26}\]
Here $\delta'(\epsilon) \rightarrow 0$ as $\epsilon \rightarrow 0$. Next, since $\eta_1 \leq {2 \eta' \over 1 + 2\eta'}$, we have
\[\int_{[2^{-j}, 2^{-j+1}] \times [2^{-k}, 2^{-k+1}]} |H_i(x,y) x^2 y^2 |^{-{\eta_1 \over 2}}\leq C \int_{[2^{-j}, 2^{-j+1}] 
\times [2^{-k}, 2^{-k+1}]} |H_i(x,y) x^2 y^2 |^{-{ \eta' \over 1 + 2\eta'}} \tag{4.27}\]
We apply Holder's inequality in the right-hand integral $(4.27)$, using exponent $1 + 2\eta'$ on $|H_i(x,y)|^{-{ \eta' \over 1 + 2\eta'}}$ and exponent 
${1 + 2\eta' \over 2\eta'}$ on $|x^2y^2|^{-{ \eta' \over 1 + 2\eta'}}$. The result is
\[\int_{[2^{-j}, 2^{-j+1}] \times [2^{-k}, 2^{-k+1}]} |H_i(x,y) x^2 y^2 |^{-{ \eta' \over 1 + 2\eta'}}\]
\[ \leq \bigg(\int_{[2^{-j}, 2^{-j+1}] 
\times [2^{-k}, 2^{-k+1}]}|H_i(x,y)|^{-\eta'}\bigg)^{1 \over 1 + 2\eta'}\bigg(\int_{[2^{-j}, 2^{-j+1}] \times [2^{-k}, 2^{-k+1}]}  (xy)^{-1}\bigg)^{2 \eta' 
\over 1 + 2\eta'} \tag{4.28}\]
The right-hand factor in $(4.28)$ is immediately uniformly bounded, and the definition of $\eta'$ is exactly that the left hand factor of $(4.28)$
is uniformly bounded in $i,j$, and $k$. Thus $(4.26)-(4.28)$ imply that 
\[||V_{ijkn}^1 ||_{L^{{2 \over \eta_1} + \delta'(\epsilon)} \rightarrow L^{{2 \over \eta_1} + \delta'(\epsilon)}_{\eta_1 - \eta_1\epsilon}}  \leq C_{\epsilon}''2^{-\delta'(\epsilon)n}  \tag{4.29}\]
Since $\epsilon$ can be made arbitrarily small and $\delta(\epsilon), \delta'(\epsilon) \rightarrow 0$ as $\epsilon \rightarrow 0$, $(4.29)$ give the
needed $L^p$ to $L^p_s$ estimates for $V_{ijkn}^1$ and we are done.

\subsection{$L^p$ to $L^p_s$ bounds for the operators $W_{ijkn}^1$}

We look at $(4.8)$ in the $l = 1$ case,  which for the reader's convenience is given by
\[\widehat{\mu_{ijkn}^1} (\lambda) = \sigma(2^{-n} \lambda)\int e^{-i\lambda_1 x  -  i\lambda_2y  - i\lambda_2 h_i(x) - i\lambda_3R(x,y)}\gamma_{ijk}(x,y)\alpha_1\big(|\lambda|^{- {\epsilon_0 \over 4}}2^{-k} P_y(x,y)\big) \,dx\,dy \tag{4.30}\]
Observe that since we are in case 2, we have $2^{-k} \leq |\lambda|^{-{1 \over 2} + {\epsilon_0 \over 8}}(R_{yy}^*)^{-{1 \over 2}}$ and therefore in the domain of integration of $(4.30)$ we have
\[1 \leq C{1 \over |\lambda|^{{1 \over 2} - {\epsilon_0 \over 8}}(R_{yy}^*)^{{1 \over 2}}y}\tag{4.31}\]
Squaring this, inserting into $(4.30)$, then taking absolute values of the integrand and integrating, gives
\[\widehat{\mu_{ijkn}^1} (\lambda) \leq C'' |\lambda|^{\epsilon_0 \over 4} \int_{[2^{-j}, 2^{-j+1}] \times [2^{-k}, 2^{-k+1}]}
{1 \over |\lambda R_{yy}^*| y^2}\,dx\,dy\tag{4.32}\]
Next, observe that due to the $\alpha_1\big(|\lambda|^{- {\epsilon_0 \over 4}}2^{-k} P_y(x,y)\big)$ factor in $(4.30)$ and the fact that 
$\partial_x P_y(x,y) = \lambda_3 R_{xy}(x,y)$, for a given $y$ the $x$-diameter of
the set where the integrand in $(4.30)$ is nonzero is at most $C|\lambda|^{-1 + {\epsilon_0 \over 4}}2^k (R_{xy}^*)^{-1}$. 
Thus the overall integral is bounded by a constant times $2^{-k}$ times this, or $C'|\lambda|^{-1 + {\epsilon_0 \over 4}} (R_{xy}^*)^{-1}$. So in
 analogy with $(4.32)$ we have
\[\widehat{\mu_{ijkn}^1} (\lambda) \leq C'' |\lambda|^{\epsilon_0 \over 4} \int_{[2^{-j}, 2^{-j+1}] \times [2^{-k}, 2^{-k+1}]}
{1 \over |\lambda R_{xy}^*| xy}\,dx\,dy\tag{4.33}\]
Suppose in the final coordinates $R(x,y) \sim x^{\alpha} y^{\beta}$ where $\beta \geq 1$. Then estimates $(4.32)$ and $(4.33)$ are all that we need. 
We use $(4.32)$ if $\beta > 1$ and $(4.33)$ if $\beta = 1$ and we obtain
\[\widehat{\mu_{ijkn}^1} (\lambda) \leq C''' |\lambda|^{\epsilon_0 \over 4} \int_{[2^{-j}, 2^{-j+1}] \times [2^{-k}, 2^{-k+1}]}
{1 \over |\lambda| x^{\alpha}y^{\beta}}\,dx\,dy\tag{4.34}\]
Since $\epsilon_0 > 0$ was arbitrary, $|\lambda| \sim 2^n$, and $U_{ijkn}^1 = f \ast \mu_{ijkn}^1$, this can be restated as the statement that for every $\epsilon > 0$ there
is a $\delta(\epsilon) > 0$ such that 
\[||U_{ijkn}^1||_{L^2 \rightarrow L^2_{1 - \epsilon}} \leq C2^{-\delta(\epsilon) n}  \int_{[2^{-j}, 2^{-j+1}] \times [2^{-k}, 2^{-k+1}]}
{1 \over x^{\alpha}y^{\beta}}\,dx\,dy\tag{4.35}\]
The desired $L^p$ to $L^p_s$ estimates for the $\beta \geq 1$ situation are now proved using an interpolation argument very similar to that of the end
of section 4.4. Namely, due to the location of the support of $\gamma_{ijk}$, the measure $\mu_{ijkn}^1$ has $L^1$ norm bounded by $C2^{-j - k} 
= C\int_{[2^{-j}, 2^{-j+1}] \times [2^{-k}, 2^{-k+1}]}({1 \over x^{\alpha}y^{\beta}})^0$. Thus by Young's inequality we have
\[||U_{ijkn}^1 f||_{L^{(\epsilon^{-1})} \rightarrow L^{({\epsilon}^{-1})}} \leq C \int_{[2^{-j}, 2^{-j+1}] \times [2^{-k}, 2^{-k+1}]}\bigg({1 \over x^{\alpha}y^{\beta}}\bigg)^0 \,dx\,dy \tag{4.36}\] 
Using the same weighting as before, in analogy with $(4.26)$, for $\eta_1 = \min(\eta, {2 \eta' \over 1 + 2\eta'})$ we have
\[||U_{ijkn}^1 f||_{L^{{2 \over \eta_1} + \delta'(\epsilon)} \rightarrow L^{{2 \over \eta_1} + \delta'(\epsilon)}_{\eta_1 - \eta_1\epsilon}}  \leq C_{\epsilon}'2^{-\delta'(\epsilon)n} \int_{[2^{-j}, 2^{-j+1}] \times [2^{-k}, 2^{-k+1}]}
\bigg({1 \over x^{\alpha}y^{\beta}}\bigg)^{{\eta_1}}\,dx\,dy \tag{4.37}\]
Here $\delta'(\epsilon) \rightarrow 0$ as $\epsilon \rightarrow 0$. Next, since $\eta_1 \leq \eta$, we have
\[  \int_{[2^{-j}, 2^{-j+1}] \times [2^{-k}, 2^{-k+1}]} \bigg({1 \over x^{\alpha}y^{\beta}}\bigg)^{\eta_1} \leq C'' \int_{[2^{-j}, 2^{-j+1}] \times [2^{-k}, 2^{-k+1}]} \bigg({1 \over x^{\alpha}y^{\beta}}\bigg)^{\eta} \,dx\,dy\tag{4.38}\]
The index $\eta$ is defined to the supremum of the $e$ for which $|S(x,y)|^{-e}$ is integrable on a neighborhood of the origin, and this index is
invariant under the coordinate changes of this paper. Hence since $|R_i(x,y)| \sim  x^{\alpha}y^{\beta}$, the integrals on the right of $(4.38)$ are
 uniformly bounded in $j$ and $k$. Thus $(4.37)$ and $(4.38)$ imply
\[||U_{ijkn}^1||_{L^{{2 \over \eta_1} + \delta'(\epsilon)} \rightarrow L^{{2 \over \eta_1} + \delta'(\epsilon)}_{\eta_1 - \eta_1\epsilon}}  \leq C_{\epsilon}''2^{-\delta'(\epsilon)n}  \tag{4.39}\]
Since $\epsilon$ can be made arbitrarily small and $\delta(\epsilon), \delta'(\epsilon) \rightarrow 0$ as $\epsilon \rightarrow 0$, $(4.39)$ gives the
needed $L^p$ to $L^p_s$ estimates for $U_{ijkn}^1$ and we are done for the case where $\beta \geq 1$.

What remains is to consider the $\beta = 0$ case, so that $|R(x,y)| \sim x^{\alpha}$. We will make use of $(4.32)$ when we are  in the second case of Theorem 3.4, since it works best in conjunction with Theorem 3.5, and we will make use of $(4.33)$ when we are in 
the first case of Theorem 3.4 since that equation works best in conjunction with lower bounds of the first case of Theorem 3.4.  Suppose we are in the first case. Then in our current notation, in this case there exists a constant $C$ and $s_i$, $g_i > 0$ such that
 $|{\partial^2 R \over\partial x\partial y} (x,y)| > C^{-1}x^{s_i - g_i - 1}$ on the support of $\gamma_{ijk}$, and
such that $|R(x,y)| < C x^{s_i}$ on a product of intervals $I_1 \times I_2  \subset [2^{-j}, 2^{-j+1}]  \times [C^{-1} 2^{-jg_i}, C2^{-jg_i}]$,
where $|I_1| > C^{-1} 2^{-j}$, $|I_2| > C^{-1}2^{-jg_i}$, and $k \geq jg_i$. By $(4.33)$ we have
\[\widehat{\mu_{ijkn}^1} (\lambda) \leq C |\lambda|^{\epsilon_0 \over 4} \int_{[2^{-j}, 2^{-j+1}] \times [2^{-k}, 2^{-k+1}]}
{1 \over |\lambda| x^{s_i - g_i }y}\,dx\,dy\tag{4.40}\]
By the steps of the above interpolation argument leading up to $(4.37)$, using the fact that $\eta_1 \leq \eta$, we have
\[||U_{ijkn}^1 f||_{L^{{2 \over \eta_1} + \delta'(\epsilon)} \rightarrow L^{{2 \over \eta_1} + \delta'(\epsilon)}_{\eta_1 - \eta_1\epsilon}}  \leq C_{\epsilon}'2^{-\delta'(\epsilon)n} \int_{[2^{-j}, 2^{-j+1}] \times [2^{-k}, 2^{-k+1}]}
\bigg({1 \over  x^{s_i - g_i  }y}\bigg)^{{\eta}} \tag{4.41}\]
Since $\eta \leq 1$, the integral in $(4.41)$ is nonincreasing in $k$. As a result, since $k \geq jg_i$, $(4.41)$ implies
\[||U_{ijkn}^1 f||_{L^{{2 \over \eta_1} + \delta'(\epsilon)} \rightarrow L^{{2 \over \eta_1} + \delta'(\epsilon)}_{\eta_1 - \eta_1\epsilon}}  \leq C_{\epsilon}'2^{-\delta'(\epsilon)n} \int_{[2^{-j}, 2^{-j+1}] \times [2^{-jg_i}, 2^{-jg_i+1}]}
\bigg({1 \over  x^{s_i - g_i  }y}\bigg)^{{\eta}} \tag{4.42}\]
But since $y \sim x^{g_i}$ in the domain of $(4.42)$, the integrand in $(4.42)$ is comparable to $(x^{s_i})^{-\eta}$. But since we are in case 
1 of Theorem 3.4, $|R(x,y)| < C x^{s_i}$ on a product of intervals $I_1 \times I_2$ whose measure is comparable to that of the domain of
integration of $(4.42)$. So the right hand side of $(4.42)$ is bounded by
$C_{\epsilon}''2^{-\delta'(\epsilon)n} \int_{I_1 \times I_2}|R(x,y)|^{-{\eta}}$.
Since $\eta$ is defined as the supremum of the $e$ for which $\int|R(x,y)|^{-e}$ is finite on a neighborhood of the origin, this
means that the integral in $(4.42)$ is uniformly bounded. Thus once again we have the desired estimate $(4.39)$.

We now move to the second case of Theorem 3.4. We raise $(4.31)$ to the ${2 \over 3}$ power and insert it in $(4.30)$, obtaining that
$|\widehat{\mu_{ijkn}^1} (\lambda)|$ is bounded by
\[ C |\lambda|^{\epsilon_0 \over 12} \int\bigg|\int {1 \over (|\lambda|R_{yy}^* y^2)^{2 \over 3}}e^{-i\lambda_1 x  -  i\lambda_2y  - i\lambda_2 h_i(x) - i\lambda_3R(x,y)}\gamma_{ijk}(x,y)\alpha_1\big(|\lambda|^{- {\epsilon_0 \over 4}}2^{-k} P_y(x,y)\big)\,dx\bigg|\,dy \tag{4.43}\]
We now use the Van der Corput lemma for either second or third derivatives in the $x$ integration in $(4.43)$. Using $(3.1)$, for 
either $m =2$ or $m = 3$ the phase
$P(x,y)$ will satisfy $|\partial_x^m P(x,y)| \geq C|\lambda|x^{\alpha - m}$. (Always $\alpha > 2$ in the degenerate case at hand). We don't have
to worry about the possibility that $h_i(x)$ has a zero of order $\alpha$ at $x = 0$; even if it does, since we can assume $|\lambda_3| > D|\lambda_2|$ 
for any preselected constant $D$ in our initial decomposition $\mu_n = \mu_n^1 + \mu_n^2$, we will still have  
$|\partial_x^m P(x,y)| \geq C|\lambda| x^{\alpha - m}$ for $m = 2$ and $3$ in this situation.

We insert the $m = 3$ case of the Van der Corput lemma into $(4.43)$ since that gives the weaker estimate. The result can be written as 
\[|\widehat{\mu_{ijkn}^1} (\lambda)| \leq  C |\lambda|^{\epsilon_0 \over 12} \int_{[2^{-j}, 2^{-j+1}] \times [2^{-k}, 2^{-k+1}]}
 {1 \over (|\lambda|R_{yy}^* y^2)^{2 \over 3}} \min(1, (|\lambda|x^{\alpha})^{-{1 \over 3}})\,dx\,dy \]
We remove the left part of the minimum and just write this as 
\[|\widehat{\mu_{ijkn}^1} (\lambda)| \leq  C |\lambda|^{\epsilon_0 \over 12} \int_{[2^{-j}, 2^{-j+1}] \times [2^{-k}, 2^{-k+1}]}
 {1 \over |\lambda|(R_{yy}^* y^2)^{2 \over 3} (x^{\alpha})^{{1 \over 3}}}\,dx\,dy \tag{4.44}\]
The interpolation argument we have been using then implies the following analogue of $(4.42)$.
\[||U_{ijkn}^1 f||_{L^{{2 \over \eta_1} + \delta'(\epsilon)} \rightarrow L^{{2 \over \eta_1} + \delta'(\epsilon)}_{\eta_1 - \eta_1\epsilon}}  \leq C_{\epsilon}'2^{-\delta'(\epsilon)n} \int_{[2^{-j}, 2^{-j+1}] \times [2^{-k}, 2^{-k+1}]}
\bigg({1 \over |\lambda|(R_{yy}^* y^2)^{2 \over 3} (x^{\alpha})^{{1 \over 3}}}\bigg)^{{\eta}} \tag{4.45}\]
By Theorem 3.5, the integral in $(4.45)$ is uniformly bounded in $j$ and $k$. Thus we once again have the desired estimate $(4.39)$.

We have now exhausted all cases, and we see that $(4.39)$ is always satisfied. Hence we are done with the analysis of the operators $U_{ijkn}^1$, and
therefore the proof of Theorem 1.1.
 
\section{References.}

\noindent [AGuV] V. Arnold, S. Gusein-Zade, A. Varchenko, {\it Singularities of differentiable maps},
Volume II, Birkhauser, Basel, 1988. \parskip = 4pt\baselineskip = 3pt

\noindent [C1] M. Christ, {\it Failure of an endpoint estimate for integrals along curves} in Fourier analysis and partial differential equations (Miraflores de la Sierra, 1992), 163-168, Stud. Adv. Math., CRC, Boca Raton, FL, 1995. 

\noindent [DZ] S. Dendrinos, E. Zimmermann, {\it On $L^p$-improving for averages associated to mixed homogeneous polynomial hypersurfaces in $\R^3$},
 J. Anal. Math. {\bf 138} (2019), no. 2, 563-595.

\noindent [Du] Duistermaat, J. J., {\it Oscillatory integrals, Lagrange immersions and unfolding of singularities},
Comm. Pure Appl. Math., {\bf 27} (1974), 207-281.

\noindent [FGoU1] E. Ferreyra, T. Godoy, M. Urciuolo, {\it Boundedness properties of some convolution operators with singular measures}, Math. Z. 
{\bf 225} (1997), no. 4, 611-624.

\noindent [FGoU2] E. Ferreyra, T. Godoy, M. Urciuolo, {\it Sharp $L^p$-$L^q$ estimates for singular fractional integral operators.}
Math. Scand. {\bf 84} (1999), no. 2, 213-230. 

\noindent [Gra] L. Grafakos, {\it Endpoint bounds for an analytic family of Hilbert transforms}, Duke Math. J. {\bf 62} (1991), no. 1, 23-59. 

\noindent [G1] M. Greenblatt, {\it $L^p$ Sobolev regularity of averaging operators over hypersurfaces and the Newton polyhedron}, J. Funct. Anal. 
{\bf 276} (2019), no. 5, 1510-1527.

\noindent [G2] M. Greenblatt, {\it Smooth and singular maximal averages over 2D hypersurfaces and associated Radon transforms}, Adv. Math. 
{\bf 377} (2021), Paper No. 107465, 45 pp.

\noindent [G3] M. Greenblatt, {\it Smoothing theorems for Radon transforms over hypersurfaces and related operators}, Forum Math. {\bf 32} (2020), 
no. 6, 1637-1647.

\noindent [G4] M. Greenblatt, {\it Convolution kernels of 2D Fourier multipliers based on real analytic functions}, J. Geom. Anal. {\bf 28} (2018), no. 2, 
787-816.

\noindent [G5] M. Greenblatt, {\it Newton polygons and local integrability of negative powers of smooth functions in the plane}, Trans. Amer. Math. Soc. 
{\bf 358} (2006), no. 2, 657-670. 

\noindent [G6] M. Greenblatt, {\it Hyperplane integrability conditions and smoothing for Radon transforms}, J. Geom. Anal. {\bf 31} (2021), no. 4, 3683-3697. 

\noindent [Gr] P. T. Gressman, {\it Uniform Sublevel Radon-like Inequalities}, J. Geom. Anal. {\bf 23} (2013), no. 2,
611-652.

\noindent [HeHoY1]  Y. Heo, S. Hong, C.W. Yang, {\it $L^p$-Sobolev regularity for integral operators over certain hypersurfaces},
 Bull. Korean Math. Soc. {\bf 51} (2014), no. 4, 965-978.

\noindent [HeHoY2] Y. Heo, S. Hong, C.W. Yang, {\it Sobolev estimates for averaging operators over a convex hypersurface in $R^3$},
 J. Math. Anal. Appl. {\bf 412} (2014), no. 1, 244-268. 

\noindent [IkKM] I. Ikromov, M. Kempe, and D. M\"uller, {\it Estimates for maximal functions associated
to hypersurfaces in $\R^3$ and related problems of harmonic analysis}, Acta Math. {\bf 204} (2010), no. 2,
151--271.

\noindent [ISa] A. Iosevich and E. Sawyer, {\it Sharp $L^p$ to $L^q$ estimates for a class of averaging operators}, Ann. Inst. Fourier, Grenoble
{\bf 46} (1996), no. 5, 1359-1384.

\noindent [ISaSe] A. Iosevich, E. Sawyer, A. Seeger, {\it On averaging operators associated with convex hypersurfaces of finite type},
 J. Anal. Math. {\bf 79} (1999), 159-187.
 
\noindent [Ka1] V. N. Karpushkin, {\it A theorem concerning uniform estimates of oscillatory integrals when
the phase is a function of two variables}, J. Soviet Math. {\bf 35} (1986), 2809-2826.

\noindent [Ka2] V. N. Karpushkin, {\it Uniform estimates of oscillatory integrals with parabolic or 
hyperbolic phases}, J. Soviet Math. {\bf 33} (1986), 1159-1188.
 
 \noindent [L], W. Littman, {\it $L^p$-$L^q$ estimates for singular integral operators arising from hyperbolic equa-
tions}, Partial differential equations. Proc. Sympos. Pure Math. {\bf 23} (1973), 479-481.

\noindent [O] D. M. Oberlin, {\it Convolution with measures on hypersurfaces}, Math. Proc. Camb. Phil. Soc.
{\bf 129} (2000), no. 3, 517-526.

\noindent [R1] B. Randol, {\it On the Fourier transform of the indicator function of a planar set}, Trans. Amer. Math. Soc. {\bf 139} 1969 271-278.

\noindent [R2] B. Randol, {\it On the asymptotic behavior of the Fourier transform of the indicator function of a convex set}, Trans. Amer. Math. Soc. 
{\bf 139} 1969 279-285. 
 
\noindent [Sc] H. Schulz, {\it Convex hypersurfaces of finite type and the asymptotics of their Fourier transforms}, Indiana Univ. Math. J. {\bf 40} 
(1991), no. 4, 1267-1275.

\noindent [Sch] J. Schwend, {\it Near Optimal $L^p \rightarrow L^q$ Estimates for Euclidean Averages Over Prototypical Hypersurfaces in $\R^3$}, 
preprint, arXiv:2012.15789. 

\noindent [Se] A. Seeger, {\it Radon transforms and finite type conditions}, J. Amer. Math. Soc. {\bf 11} (1998), no. 4, 869-897.

\noindent [S1] E. M. Stein, {\it Harmonic analysis; real-variable methods, orthogonality, and oscillatory \hfill\break
integrals}, Princeton Mathematics Series {\bf 43}, Princeton University Press, Princeton, NJ, 1993.

\noindent [S2] E. M. Stein, {\it $L^p$ boundedness of certain convolution operators}, Bull. Amer. Math. Soc. {\bf 77} (1971), 404-405.

\noindent [St] B. Street, {\it Sobolev spaces associated to singular and fractional Radon transforms}, Rev. Mat. Iberoam. {\bf 33} (2017), no. 2, 633-748.

\noindent [Str] R. Strichartz, {\it Convolutions with kernels having singularities on the sphere}, Trans. Amer. Math. Soc. {\bf 148} (1970), 461-471.
 
\noindent [V] A. N. Varchenko, {\it Newton polyhedra and estimates of oscillatory integrals}, Functional 
Anal. Appl. {\bf 18} (1976), no. 3, 175-196.

\vskip 0.3 in 

\noindent Department of Mathematics, Statistics, and Computer Science \hfill \break
\noindent University of Illinois at Chicago \hfill \break
\noindent 322 Science and Engineering Offices \hfill \break
\noindent 851 S. Morgan Street \hfill \break
\noindent Chicago, IL 60607-7045 \hfill \break
\noindent greenbla@uic.edu

\end{document}